\documentclass[12pt]{amsart}
\usepackage[margin=1in]{geometry}
\usepackage[english]{babel}
\usepackage[utf8]{inputenc}
\usepackage{subcaption}
\usepackage{amsmath}
\usepackage{amssymb}
\usepackage{amsfonts}
\usepackage{amsthm}
\usepackage{mathdots}
\usepackage{mathrsfs}
\usepackage[all]{xy}
\usepackage[pdftex]{graphicx}
\usepackage{cite}
\usepackage{url}
\usepackage{tikz}
\usepackage{hyperref}
\usepackage[T1]{fontenc} 
\usepackage{pgfplots}
\pgfplotsset{compat=1.18}

\usepgfplotslibrary{fillbetween}

\usetikzlibrary{decorations.pathreplacing,calligraphy}
\usetikzlibrary{calc}

\usepackage{hyperref,color}
\usepackage[capitalize,nameinlink]{cleveref}
\usepackage{xcolor}
\hypersetup{
  colorlinks   = true, %Colours links instead of ugly boxes
  urlcolor     = pink, %Colour for external hyperlinks
  linkcolor    = blue, %Colour of internal links
  citecolor   = green %Colour of citations
}
\usepackage{appendix}
\crefname{appsec}{Appendix}{Appendices}

\newif\iflinecheck
\linecheckfalse

\DeclareMathOperator{\PP}{\mathbf{P}}
\DeclareMathOperator{\EE}{\mathbb{E}}
\DeclareMathOperator{\ZZ}{\mathbb{Z}}
\DeclareMathOperator{\RR}{\mathbb{R}}
\DeclareMathOperator{\Lap}{\mathcal{L}}

\DeclareMathOperator{\mcM}{\mathcal{M}}
\DeclareMathOperator{\mcP}{\mathcal{P}}

\DeclareMathOperator{\mcMG}{\mathcal{M}_G}

\DeclareMathOperator{\Beta}{\mathbf{B}}

\DeclareMathOperator{\Ein}{Ein}

\newcommand{\linechecktext}[1]{\iflinecheck #1 \else \fi}

\newcommand{\df}[1]{\, \mathrm{d}#1}
\newcommand{\ddf}[1]{\frac{\mathrm{d}}{\mathrm{d}#1}}
\newcommand{\dddf}[2]{\frac{\mathrm{d}#1}{\mathrm{d}#2}}

\newcommand{\eps}{\varepsilon}

\newcommand{\ol}{\overline}
\renewcommand{\P}{\mathbf{P}}

\newtheorem{theorem}{Theorem}[section]
\newtheorem{corollary}[theorem]{Corollary}
\newtheorem{proposition}[theorem]{Proposition}

\newtheorem{lemma}{Lemma}[theorem]

\theoremstyle{definition}
\newtheorem{definition}[theorem]{Definition}

\theoremstyle{remark}
\newtheorem{remark}[theorem]{Remark}
\newtheorem{example}[theorem]{Example}

\newtheorem{obs}[theorem]{Observation}

\numberwithin{equation}{section}

\title{Multidisperse Random Sequential Adsorption and Generalizations}
\author{Roger Fan, Nitya Mani}

\begin{document}

\begin{abstract}
    In this paper, we present a unified study of the limiting density in \\one-dimensional random sequential adsorption (RSA) processes where segment lengths are drawn from a given distribution. In addition to generic bounds, we are also able to characterize specific cases, including multidisperse RSA, in which we draw from a finite set of lengths, and power-law RSA, in which we draw lengths from a power-law distribution.
\end{abstract}

\maketitle

\section{Introduction}

\subsection{Background} The field of \textit{random sequential adsorption} (RSA) studies processes in which particles are sequentially adsorbed onto a substrate such that the particles do not overlap. Known also as \emph{simple sequential inhibition}, \emph{on-line packing}, and the \emph{hard-core model}, RSA is a fundamental process that has been extensively studied in mathematics and statistical physics. RSA also has many applications to biological and chemical processes: for example, reactions on polymer chains have been modelled with RSA, along with various chemisorption (chemical adsorption) processes. See \cite{evanssurvey,privmansurvey} for surveys of RSA and its applications.

Meanwhile, the theoretical study of RSA has proven difficult. There is not much quantitative theory for RSA in more than one dimension (some notable results include~\cite{penrose2001,penrose2002}); even in one dimension, only a small number of RSA processes are rigorously understood, despite the large amount of interest in them. In statistical physics, many papers study RSA heuristically, often employing large scale Monte-Carlo simulations \cite{araujo2010jammed,hassan2002}.

The mathematical study of RSA began with the \emph{R\'enyi parking problem}, proposed by Alfred R\'enyi \cite{Renyi} in 1958. On the interval $[0,L]$, we randomly park a length-$1$ segment by choosing its left endpoint uniformly at random from $[0,L-1]$. Then, we randomly park a second length-$1$ segment. If the second segment intersects the first, we discard it and randomly choose a new length-$1$ segment until it does not overlap with the first. In this way, we continue parking length-$1$ segments by randomly choosing ``parking spots'' on $[0,L]$ that we discard if they would cause the segment to overlap with any previously parked segments. We repeat this process until the gaps between segments are too small for any more segments to be parked, at which point we say the process has reached \emph{saturation}. 

\begin{figure}\centering
    \begin{tikzpicture}
        \draw[black] (-5,0) -- (5,0);
        \filldraw [black] (-5,0) circle (3pt);
        \filldraw [black] (5,0) circle (3pt);
    \end{tikzpicture}
    \begin{tikzpicture}
        \draw[black] (-5,0) -- (5,0);
        \filldraw [black] (-5,0) circle (3pt);
        \filldraw [black] (5,0) circle (3pt);
        \filldraw [color=black, fill=gray] (-1.1,-.1) rectangle (-3.1,.1);
    \end{tikzpicture}
    \begin{tikzpicture}
        \draw[black] (-5,0) -- (5,0);
        \filldraw [black] (-5,0) circle (3pt);
        \filldraw [black] (5,0) circle (3pt);
        \filldraw [color=black, fill=gray] (-1.1,-.1) rectangle (-3.1,.1);
        \filldraw [color=black, fill=gray] (0.5,-.1) rectangle (2.5,.1);
    \end{tikzpicture}
    \begin{tikzpicture}
        \draw[black] (-5,0) -- (5,0);
        \filldraw [black] (-5,0) circle (3pt);
        \filldraw [black] (5,0) circle (3pt);
        \filldraw [color=black, fill=gray] (-1.1,-.1) rectangle (-3.1,.1);
        \filldraw [color=black, fill=gray] (0.5,-.1) rectangle (2.5,.1);
        \filldraw [color=black, fill=gray] (2.8,-.1) rectangle (4.8,.1);
    \end{tikzpicture}
    \caption{R\'enyi's parking problem.}
\end{figure}
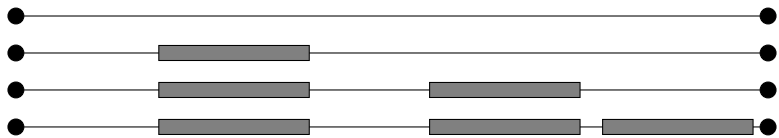

R\'enyi studied the number of segments parked at \textit{saturation}, given by the random variable $N_L$, deriving an integral recurrence equation for its expected value:
\[ \EE[N_{L+1}] = 1 + 2 \int_0^L \EE[N_t] \df{t}. \]
By studying the Laplace transform, he was able to characterize the limiting density $\EE[N_L]/L$ as $L \to \infty$.
\begin{theorem}[R\'enyi]
    Let $N_L$ be defined as above. Then,
    \[ \lim_{L \rightarrow \infty} \frac{\EE[N_L]}{L} = \alpha, \]
    where $\alpha$ is the \emph{R\'enyi parking constant}, given by
    \[ \alpha = \int_0^\infty \exp \left( -2\int_0^t \frac{1-e^{-u}}{u} \df{u} \right) \df{t} \approx 0.747598. \]
\end{theorem}
The asymptotics of $\EE[N_L]$ are known to much greater precision~\cite{DR64}, along with a similar asymptotic characterization of the variance of $N_L$~\cite{DR64,MAN64}. Together with fairly general central limit theorems~\cite{penrose2002,penrose2001,penrose03}, this body of work characterizes the limiting distribution of $N_L$ in one-dimensional random sequential adsorption.

In this paper, we study generalizations of R\'enyi's parking problem in which we park segments of varying lengths. These processes are also known as \emph{cooperative} and \emph{competitive RSA}~\cite{burridge2004recursive, evanssurvey,subashiev2007}.
One of the most highly studied process we investigate in this work is \emph{multidisperse RSA}, in which there are $n$ different possible segment lengths, each with a unique probability of arriving at a given time. In other words, when we randomly choose a segment to park, we first randomly choose its length from the $n$ different lengths according to some underlying distribution. Then, we choose a random location for the segment on the interval by choosing its left endpoint uniformly at random from $[0,L]$. If the segment is not fully contained within the interval $[0,L]$, or if it overlaps with another segment, we discard it and choose a new segment with a new random length and position. Multidisperse RSA has been widely studied in statistical physics \cite{subashiev2007, reeve2015random,hassan2002,araujo2010jammed} and is motivated by many chemisorption processes in which a mixture of chemicals are absorbed onto a substrate \cite{ligands,phosphate}.

\subsection{Main results}
Our first main result concerns \textit{multidisperse RSA} and generalizes work of~\cite{subashiev2007} who studied the setting of two segment lengths. For any collection of $n$ segment lengths and an associated distribution, we show that each segment length is expected to cover a linear fraction of the interval $[0, L]$ at saturation. Further, in~\cref{multi-constant}, we give an exact analytical expression for the limiting fraction of $[0, L]$ covered by segments of each length at saturation.

We also study a much vaster generalization of the R\'enyi parking problem, given by choosing segment lengths from an arbitrary continuous distribution. We describe the possible lengths and their weights with a \emph{length distribution function} (ldf) $\nu : [1,\infty) \rightarrow \RR$, which describes the relative probabilities of choosing each possible segment length. When $\int_1^\infty \nu(\ell) \df{\ell} < \infty$, we call $\nu$ \emph{convergent}, and when the integral diverges, we call $\nu$ \emph{divergent}. Intuitively, convergent ldfs tend to weight small segment lengths more, whereas divergent ldfs tend to weight large segment lengths more (c.f. \cref{diff-rsa-fig}).

\begin{figure}\centering
    \begin{tikzpicture}
        \draw[black] (-8,0) -- (8,0);
        \filldraw [black] (-8,0) circle (3pt);
        \filldraw [black] (8,0) circle (3pt);

        \filldraw [color=black, fill=gray] (0.7386850905421304,-0.05) rectangle (3.8986850905421315,0.05);
        \filldraw [color=black, fill=gray] (-3.155340777610611,-0.05) rectangle (0.00465922239,0.05);
        \filldraw [color=black, fill=gray] (-5.548885041431431,-0.05) rectangle (-4.495551708098098,0.05);
        \filldraw [color=black, fill=gray] (-6.879045259325319,-0.05) rectangle (-5.825711925991985,0.05);
        \filldraw [color=black, fill=gray] (-4.426635850385518,-0.05) rectangle (-3.3733025170521844,0.05);
        \filldraw [color=black, fill=gray] (4.619589554566079,-0.05) rectangle (5.6729228878994125,0.05);
        \filldraw [color=black, fill=gray] (5.982051444566609,-0.05) rectangle (7.035384777899942,0.05);
    \end{tikzpicture}
    {\small Multidisperse RSA with lengths 1 and 3.}
    \begin{tikzpicture}
        \draw[black] (-8,0) -- (8,0);
        \filldraw [black] (-8,0) circle (3pt);
        \filldraw [black] (8,0) circle (3pt);

        \filldraw [color=black, fill=gray] (0.781115688028112,-0.05) rectangle (3.866463947798672,0.05);
        \filldraw [color=black, fill=gray] (-5.765963377963334,-0.05) rectangle (-4.906192412611011,0.05);
        \filldraw [color=black, fill=gray] (-7.68545982303797,-0.05) rectangle (-6.687362155610542,0.05);
        \filldraw [color=black, fill=gray] (-6.600325184329087,-0.05) rectangle (-5.792745343224063,0.05);
        \filldraw [color=black, fill=gray] (-2.057727813235157,-0.05) rectangle (-1.0759513607513367,0.05);
        \filldraw [color=black, fill=gray] (-3.436875462435156,-0.05) rectangle (-2.0732552920886675,0.05);
        \filldraw [color=black, fill=gray] (-4.561622834152155,-0.05) rectangle (-3.6841876343044935,0.05);
        \filldraw [color=black, fill=gray] (-0.49021861923938165,-0.05) rectangle (0.6198938550136031,0.05);
        \filldraw [color=black, fill=gray] (3.9749995070548523,-0.05) rectangle (4.960799762965462,0.05);
        \filldraw [color=black, fill=gray] (6.371684747855938,-0.05) rectangle (7.658260543023942,0.05);
        \filldraw [color=black, fill=gray] (4.9925914637581545,-0.05) rectangle (5.896858609067011,0.05);
    \end{tikzpicture}
    {\small RSA with convergent ldf $\nu(\ell) = \ell^{-2}$.}
    \begin{tikzpicture}
        \draw[black] (-8,0) -- (8,0);
        \filldraw [black] (-8,0) circle (3pt);
        \filldraw [black] (8,0) circle (3pt);
        
        \filldraw [color=black, fill=gray] (3.965462387397309,-0.05) rectangle (7.613389733712754,0.05);
    \filldraw [color=black, fill=gray] (-5.435263690947889,-0.05) rectangle (3.1865325591621905,0.05);
    \filldraw [color=black, fill=gray] (-6.611999272457058,-0.05) rectangle (-5.775207681063002,0.05);
    \filldraw [color=black, fill=gray] (-7.782849817297915,-0.05) rectangle (-6.809075558229519,0.05);
    \end{tikzpicture}
    {\small RSA with divergent ldf $\nu(\ell) = 1$.}
    \caption{Different RSA processes.}
    \label{diff-rsa-fig}
\end{figure}
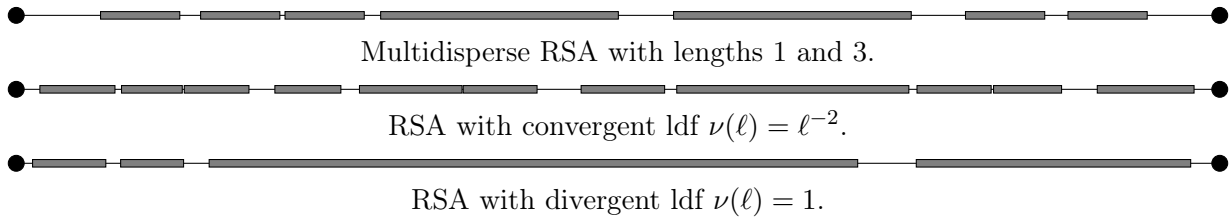

For some choices of $\nu$, we expect a $1 - o(1)$ fraction of the interval $[0, L]$ to be covered by segments. Thus, in our analysis, we typically consider $S_L$, the amount of \textit{empty space} that is \textit{not} covered by segments at saturation. We derive the following integral recurrence equation for $\EE[S_L]$:
\[ \left( \int_0^L Z_\nu(t) \df{t} \right) \EE[S_L] = 2 \int_0^{L} \EE[S_t] Z_\nu(L-t) \df{t},\]
where $Z_\nu(L) = \int_0^L \nu(\ell) \df{\ell}$ is the normalizing constant of $\nu$. By studying this recurrence, we are able to show that convergent and divergent ldfs yield fundamentally different behavior. 

In~\cref{conv-thm}, we show that when $\nu$ is convergent, under a mild additional condition, there exists some $\alpha = \alpha(\nu) > 0$ such that $\EE[S_L]/L \to \alpha$ as $L \to \infty$.
Meanwhile, in~\cref{div-thm}, we find that $\nu$ is divergent, we are able to show that (also under a mild technical growth condition) that $\EE[S_L] = o(L)$.

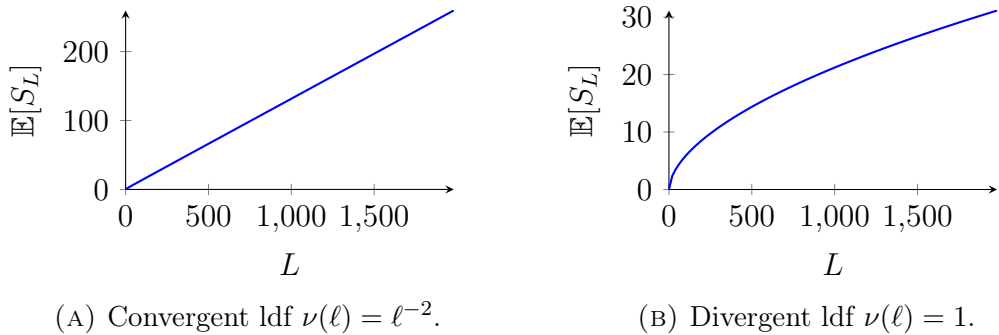
\begin{figure}
    \centering
    \begin{subfigure}{.4\textwidth}
        \begin{tikzpicture}[scale=1]
            \begin{axis} [axis lines = left,
                xlabel = \(L\),
                ylabel = {$\EE[S_L]$},
                width=.9\textwidth,
                height=.6\textwidth,
                ymin=0,
            ]
            \addplot [color=blue, thick] coordinates { 
                (0,0)(20,3.16301)(40,5.85338)(60,8.51284)(80,11.1589)(100,13.7974)(120,16.431)(140,19.0613)(160,21.689)(180,24.3148)(200,26.9391)(220,29.5621)(240,32.1841)(260,34.8052)(280,37.4256)(300,40.0453)(320,42.6645)(340,45.2832)(360,47.9014)(380,50.5193)(400,53.1368)(420,55.754)(440,58.3709)(460,60.9875)(480,63.6039)(500,66.2201)(520,68.8361)(540,71.4518)(560,74.0674)(580,76.6829)(600,79.2982)(620,81.9133)(640,84.5284)(660,87.1432)(680,89.758)(700,92.3727)(720,94.9872)(740,97.6017)(760,100.216)(780,102.83)(800,105.445)(820,108.059)(840,110.673)(860,113.287)(880,115.901)(900,118.514)(920,121.128)(940,123.742)(960,126.355)(980,128.969)(1000,131.583)(1020,134.196)(1040,136.81)(1060,139.423)(1080,142.036)(1100,144.65)(1120,147.263)(1140,149.876)(1160,152.489)(1180,155.102)(1200,157.715)(1220,160.329)(1240,162.942)(1260,165.555)(1280,168.167)(1300,170.78)(1320,173.393)(1340,176.006)(1360,178.619)(1380,181.232)(1400,183.845)(1420,186.457)(1440,189.07)(1460,191.683)(1480,194.296)(1500,196.908)(1520,199.521)(1540,202.133)(1560,204.746)(1580,207.359)(1600,209.971)(1620,212.584)(1640,215.196)(1660,217.809)(1680,220.421)(1700,223.034)(1720,225.646)(1740,228.258)(1760,230.871)(1780,233.483)(1800,236.096)(1820,238.708)(1840,241.32)(1860,243.933)(1880,246.545)(1900,249.157)(1920,251.77)(1940,254.382)(1960,256.994)(1980,259.606)
            };
            \end{axis}
        \end{tikzpicture}
        \caption{Convergent ldf $\nu(\ell) = \ell^{-2}$.}
    \end{subfigure}\qquad
    \begin{subfigure}{.4\textwidth}
        \begin{tikzpicture}[scale=1]
            \begin{axis} [axis lines = left,
                xlabel = \(L\),
                ylabel = {$\EE[S_L]$},
                width=.9\textwidth,
                height=.6\textwidth,
                ymin=0,
                xmin=0,
            ]
            \addplot [color=blue, thick] coordinates { 
                (0,0)(20,2.35927)(40,3.48121)(60,4.37094)(80,5.13704)(100,5.82264)(120,6.45023)(140,7.03335)(160,7.58094)(180,8.09923)(200,8.59282)(220,9.0652)(240,9.51909)(260,9.95668)(280,10.3797)(300,10.7897)(320,11.1879)(340,11.5753)(360,11.9529)(380,12.3213)(400,12.6813)(420,13.0336)(440,13.3785)(460,13.7167)(480,14.0484)(500,14.3742)(520,14.6942)(540,15.009)(560,15.3186)(580,15.6235)(600,15.9237)(620,16.2196)(640,16.5114)(660,16.7992)(680,17.0832)(700,17.3635)(720,17.6404)(740,17.9139)(760,18.1842)(780,18.4513)(800,18.7155)(820,18.9768)(840,19.2354)(860,19.4912)(880,19.7445)(900,19.9952)(920,20.2435)(940,20.4895)(960,20.7332)(980,20.9746)(1000,21.2139)(1020,21.4511)(1040,21.6863)(1060,21.9195)(1080,22.1508)(1100,22.3802)(1120,22.6078)(1140,22.8336)(1160,23.0577)(1180,23.2801)(1200,23.5009)(1220,23.72)(1240,23.9376)(1260,24.1537)(1280,24.3682)(1300,24.5813)(1320,24.7929)(1340,25.0032)(1360,25.2121)(1380,25.4196)(1400,25.6258)(1420,25.8307)(1440,26.0344)(1460,26.2368)(1480,26.4381)(1500,26.6381)(1520,26.837)(1540,27.0347)(1560,27.2313)(1580,27.4268)(1600,27.6212)(1620,27.8146)(1640,28.0069)(1660,28.1982)(1680,28.3884)(1700,28.5777)(1720,28.766)(1740,28.9534)(1760,29.1398)(1780,29.3253)(1800,29.5099)(1820,29.6935)(1840,29.8763)(1860,30.0583)(1880,30.2393)(1900,30.4195)(1920,30.5989)(1940,30.7775)(1960,30.9553)(1980,31.1323)
            };
            \end{axis}
        \end{tikzpicture}
        \caption{Divergent ldf $\nu(\ell) = 1$.}
    \end{subfigure}
    \caption{Growth of $\EE[S_L]$ with various length distributions.}
\end{figure}

Different RSA processes in which lengths are drawn from a distribution have been previously considered by mathematicians \cite{neydistribution,ananjevskii2016generalizations}, which our work generalizes. Moreover, apart from one or two very specific PDFs, only distributions with a maximum segment length have been studied before, but we allow arbitrarily large segment lengths, which yields a different, more complicated analysis.

Finally, we consider a class of ldfs given by power-law functions. Various specific power-law size distributions have been considered in RSA processes before \cite{power-law-circles}.~\cref{div-thm}, implies that $\EE[S_L] = o(L)$ for such power law distributions immediately. However, using more involved bounding techniques, when the distribution associated to $\nu$ follows a power law, in~\cref{plaw-asymptotics} we give asymptotically tight bounds on $\EE[S_L]$ (as $L^{f(\nu)}$ for some explicit function $f$ in the parameters of the power law distribution). We are also able to characterize the uniform length distribution as a special case of our work (c.f. \cref{puniform-asymptotic}).

\subsection{Outline}
We provide formal definitions of the processes we study and associated technical preliminaries in \cref{prelims}.

In \cref{ghost-m}, we study a related collection of processes, known as \textit{ghost} or \textit{Mat\'ern} processes, given by an appropriate thinning of classical RSA processes (see \cref{ghost-m-defn} for a formal definition). In general, ghost processes are better understood than classical RSA processes.

In \cref{multi}, we characterize the expected saturation densities of multidisperse RSA processes. In \cref{ldf}, we consider RSA with a general length distribution, and in \cref{plaw}, we study power-law length distributions.

\subsection*{Acknowledgements}
We thank the MIT PRIMES-USA research program for making this project possible. In particular, we thank Dr. Tanya Khovanova and Dr. Felix Gotti for their helpful advice throughout the paper-writing process. NM was supported by a Hertz Graduate Fellowship and the NSF GRFP.

\section{Preliminaries}\label{prelims}

\subsection{Notation}

Throughout this paper, we always park \emph{segments} on an \emph{interval}. We always use $L$ to denote the length of interval, $T$ to denote time in RSA processes, $\nu$ to denote length distribution functions (c.f. \cref{ldf-defn}), and $Z_\nu$ to denote the normalizing constant of $\nu$ (c.f. \cref{ldf-norm-defn}). We use $[n]$ to denote the set $\{1,2,\cdots,n\}$. To write that $x$ is a real number drawn uniformly at random from an interval $[a,b]$, we write $x \xleftarrow{R} [a,b]$.

We use $\RR^{>0}$ to denote the positive reals and $\RR^{\geq 0}$ to denote the nonnegative reals. We define the function $\Gamma: \RR^{> 0} \rightarrow \RR$ as the Gamma function, given by
\[ \Gamma(z) = \int_0^\infty t^{z-1}e^t \df{t}. \]

For function $f: I \to \RR$ that is $j$-times differentiable over connected interval $I$, we let $f^{(i)}(x)$ denote the $i$th derivative of $f$ with respect to $x$ for $1 \le i \le j.$

Finally, we use the following symbols for asymptotic notation (always taken as $L \rightarrow \infty$):
\begin{itemize}
    \item $f = o(g)$ and $f \ll g$ denote $\lim_{L \rightarrow \infty} f(L)/g(L) = 0$.
    \item $f(L) \sim g(L)$ denotes $\lim_{L \rightarrow \infty} f(L)/g(L) = 1$.
\end{itemize}

\subsection{The $\nu$-RSA process}

We study the $\nu$-RSA process, in which segment lengths are drawn from a distribution given by a \emph{length distribution function}. Note that if we allowed segment lengths to be any positive real, $\nu$-RSA processes with positive support on all of $(0, \eps)$ for some $\eps > 0$ would never reach saturation; at every time step, the remaining space could always fit one additional segment drawn from $\nu$. Thus, without loss of generality, we require that the minimum segment length in the distribution is $1$. More formally we have the following:

\begin{definition}\label{ldf-defn}
    A \emph{length distribution function} (or \emph{ldf}) is a nonnegative integrable function $\nu : [0,\infty) \rightarrow \RR^{\ge 0}$ such that $\nu(\ell) = 0$ for $\ell \in [0,1)$, and $\int_1^\ell \nu(t) \df{t} > 0$ for all $\ell > 1$ (the condition that the minimum segment length is equal to $1$).
\end{definition}

We will construct distributions on $[0,L]$ induced by an ldf $\nu$ on $[0,L]$. To normalize the distribution, we define a \emph{normalizing constant} for every ldf:

\begin{definition}\label{ldf-norm-defn}
    Given ldf $\nu(\ell)$, its \emph{normalizing constant} is the function $Z_\nu : \RR^{\geq 0} \rightarrow \RR^{\geq 0}$ given by
    \[ Z_\nu(L) = \int_0^L \nu(\ell) \df{\ell}. \]
\end{definition}

We also make the following distinction between \emph{convergent} and \emph{divergent} ldfs:

\begin{definition}\label{d:conv-div}
    We say an ldf $\nu(\ell)$, with normalizing constant $Z_\nu$, is \emph{convergent} if
    \[ \lim_{L \rightarrow \infty} Z_\nu (L) < \infty. \]
    Otherwise, $\nu(\ell)$ is \emph{divergent}.
\end{definition}

Given an ldf, we describe an RSA process in which segment lengths are drawn from the ldf truncated at $L$. We call this the \emph{$\nu$-RSA process}, formally defined as follows:

\begin{definition}[The $\nu$-RSA process]\label{RSA-defn}
    Let $\nu(\ell)$ be an ldf. Then, let the \emph{$\nu$-RSA process} be the following stochastic process, in which we attempt to park segments on an interval of length $L$ where the segment lengths drawn are from $\nu(\ell)$:
    
    Initialize:
    \begin{itemize}
        \item $I_0 = [0, L],$ the empty region not occupied by parked segments,
        \item $P_0 = \emptyset$, the set of parked segments.
    \end{itemize}

    Then, for $T = 1, 2, \ldots$:
    \begin{itemize}
        \item Sample $b \xleftarrow{R} [0,L]$, and choose a length $\ell \in [1,L]$ according to the probability density function $p_L(\ell) \propto \nu(\ell)$, i.e. $p_L(\ell) = \frac{\nu(\ell)}{Z_{\nu} (L) }.$
        \item If the segment $(b,b+\ell) \subseteq I_{T-1}$, let $I_T = I_{T-1} \setminus (b,b+\ell)$ and $P_T = P_{T-1} \cup \{ (b,b+\ell) \}$. We say the segment $(b,b+\ell)$ has been \emph{parked}. Otherwise, let $I_T = I_{T-1}$ and $P_T = P_{T-1}$, and we say that the segment $(b,b+\ell)$ has been \emph{rejected}.
        \item If all connected intervals in $I_T$ are of length less than $1$, we say the process is \emph{at saturation}.
    \end{itemize}

    We define the \emph{empty space} at saturation to be the random variable $S_L$, defined as the total length not covered by parked segments at saturation, viz.
    \[ S_L = \lim_{T \rightarrow \infty} \lambda(I_T), \]
    where $\lambda(I_T)$ denotes the Lesbegue measure of $I_T$.
\end{definition}

\begin{remark}\label{diff-rsa-remark}
    Ney in \cite{neydistribution} and Ananjevskii in \cite{ananjevskii2016generalizations} analyze similar processes to the one described above. However, in their processes, we first choose the segment length and \emph{then} place the segment randomly on the interval, with no possibility of rejection. Because of this, their process tends to weight large segments more than ours. Moreover, they consider length distributions with segment length bounded above, whereas we allow arbitrarily large segments, which yields a different analysis and a broader class of questions.
\end{remark}

A version of the $\nu$-RSA process is the \emph{multidisperse process} (c.f. \cref{multidisperse-defn}), in which we draw segment lengths from a discrete set of lengths $\ell_1=1, \ell_2, \cdots, \ell_n$, according to probabilities $q_1, \cdots, q_n$. The multidisperse process can be thought of as the $\nu$-RSA process with
\[ \nu(\ell) = q_1 \delta(\ell-\ell_1) + \cdots + q_n \delta(\ell-\ell_n), \]
where $\delta$ is the Dirac delta function. We will study multidisperse processes in \cref{multi}.

\subsection{Analytic Tools}

We employ the Laplace transform heavily throughout this paper, so we briefly recall its definition and some of its basic properties below.
\begin{definition}
    The \emph{Laplace transform} of a function $f : \RR^{\geq 0} \rightarrow \RR$ is the function $\Lap \{ f \} : \RR^{> 0} \rightarrow \RR$ given by
    \[ \Lap \{ f \} (s) = \int_0^\infty f(x) e^{-sx} \df{x}. \]
\end{definition}
\begin{proposition}[Laplace Transform Properties]\label{Laplace-Properties}
    Let $F(s), G(s)$ be the Laplace transforms of functions $f, g$, and let $\star$ be the convolution operator, defined as
    \[ (f \star g) (x) = \int_0^x f(t) g(x-t) \df{t} = \int_0^x f(x-t) g(t) \df{t}. \]
    The Laplace transform enjoys the following properties:
    \begin{enumerate}
        \item (Linearity) For $a, b \in \RR$, $\Lap \{ a f(x) + b g(x) \} = aF(s) + bG(s). $
        \item (Differentiation) Let $f$ be $n$-times differentiable. Moreover, let $f(0^+)$ denote the limit $\lim_{x \rightarrow 0^+} f(x)$, and let $f^{(n)}$ denote the $n$-th derivative of $f$. Then, $ \Lap \{ f^{(n)} (x) \} = s^n F(s) - \sum_{k=1}^n s^{n-k} f^{k-1} (0^+). $
        \item (Integration) $\Lap \left\{ \int_0^x f(t) \df{t} \right\} = \frac{F(s)}{s}.$
        \item (Translation) For $a \in \RR$, $ \Lap \left\{ f(x+a) \right\} = e^{as} \left( F(s) - \int_0^a f(x) e^{-sx} \df{x} \right). $
        \item (Time-Multiplication) For $n \in \ZZ^{>0}$, $ \Lap\{ x^n f(x) \} = F^{(n)} (s). $
        \item (Convolution) $\Lap \{ ( f \star g ) (x) \} = F(s) G(s).$
        \item (Abelian Final Value Theorem) If $f$ is bounded and there exists a constant $C$ for which $\lim_{x \rightarrow \infty} f(x) = C$, then $\lim_{s \rightarrow 0^+} s F(s) = C.$
    \end{enumerate}
\end{proposition}

In our analysis, we apply the Hardy-Littlewood Tauberian Theorem (c.f. \cite{TauberianTheory} p. 30), which relates the behavior of a function's Laplace transform around $0$ to the behavior of the function at infinity. This will allow us to relate various quantities related to $N_L/L$ and $S_L/L$ as $L \to \infty$ to the Laplace transforms of related functions at $0$.
\begin{theorem}[Hardy-Littlewood Tauberian Theorem]\label{Tauberian}
    If $f: \RR \rightarrow \RR$ is positive and integrable, $e^{-st} f(t)$ is integrable, and as $s \rightarrow 0$, there exist constants $H, \beta$ such that $\Lap \{ f \} (s) \sim \frac{H}{s^\beta},$
    then as $x \rightarrow \infty$,
    $\int_0^x f(t) \df{t} \sim \frac{H}{\Gamma(\beta+1)} x^{\beta}.$
\end{theorem}

We also will later require a couple technical propositions, which we will prove here. Analogous forms of \cref{diffineq} and \cref{negdiffineq} with the opposite inequality also hold.

First, we observe that relationships between the $k$-th derivatives of two functions sometimes yield inequalities between the functions.
\begin{proposition}\label{diffineq}
    Let $y, v : [a,b) \rightarrow \RR$ be $n$-times differentiable functions, and for all $k \in \{0,1,\ldots,n-1\}$, let $c_k$ be a nonnegative function $c_k : [a,b) \to \RR^{\ge 0}$. Finally, let $C$ be a function $C : [a,b) \rightarrow \RR$. If $y^{(k)} (a) = v^{(k)} (a)$ for each $k \in \{0,1,\ldots,n-1\}$, and if for all $x \in [a,b)$,
    \[ y^{(n)} (x) < C(x) + \sum_{k = 0}^{n-1} c_k(x) y^{(k)}(x), \qquad v^{(n)} (x) = C(x) + \sum_{k = 0}^{n-1} c_k(x) v^{(k)}(x), \]
    then for all $x \in [a,b)$, $y(x) \leq v(x)$.
\end{proposition}

\begin{proof}
    Consider $f(x) = v(x) - y(x)$. Note that $f^{(k)} (a) = 0$ for $k < n$, and
    \begin{equation}\label{diffineq-fineq}
        f^{(n)} (x) > \sum_{k = 0}^{n-1} c_k(x) f^{(k)}(x).
    \end{equation}

    Notably, $f^{(n)} (a) > 0$. We now claim there exists $\eps$ such that $f^{(k)} (x)$ is positive on $(a,a+\eps)$ for all $k \leq n$. This is because $f^{(k)} (x)$ satisfies
    \[ \left( f^{(k)} \right)^{(1)} (a) = \left( f^{(k)} \right)^{(2)} (a) = \cdots = \left( f^{(k)} \right)^{(n-k-1)} (a) = 0, \qquad \left( f^{(k)} \right)^{(n-k)} (a) > 0,\]
    where $\left( f^{(k)} \right)^{(i)}$ is the $i$-th derivative of $f^{(k)}$. The general derivative test on $f^{(k)}$ then implies that $f^{(k)}$ is strictly increasing in a small neighborhood $[a, a+\eps_k)$ around $a$. As $f^{(k)} (a) = 0$, this implies that $f^{(k)} (x)$ is positive on $(a, a+ \eps_k)$. Letting $\eps$ be the minimum of the $\eps_k$ gives us our desired interval $(a, a+\eps)$.

    Now, for contradiction, assume there exists $x \in [a,b)$ such that $f(x) < 0$. Then, the following infimum exists: 
    \[ m := \inf \{x \in [a,b) : f^{(k)} (x) < 0 \text{ for some } k < n \} .\]
    Note that $m \geq a + \eps$, as the $f^{(k)}$'s are nonnegative on $[a,a+\eps)$. We claim there exists some $K < n$ for which $f^{(K)} (m) \leq 0$. Otherwise, $f^{(k)} (m) > 0$ for all $k < n$, and by continuity of the $f^{(k)}$'s, there exists a sufficiently small $\eps'$ for which all $f^{(k)} (x)$ are positive on $[m,m+\eps']$, which contradicts the definition of $m$.

    Thus, let $f^{(K)} (m) \leq 0$, and let $x_K = m$. By the mean value theorem, there exists some $x_{K+1}$ in $(a,m)$ such that 
    \[ f^{(K+1)} (x_{K+1}) = \frac{f^{(K)} (x_K) - f^{(K)} (a) }{ x_K - a } = \frac{f^{(K)} (x_K) }{x_K - a} \leq 0. \]
    We may repeatedly apply the mean value theorem in this fashion, showing that there exists $x_{K+2} \in (0,m)$ for which $f^{(K+2)} (x_{K+2}) \leq 0$, and so on, until we find $x_n \in (a,m)$ for which $f^{(n)} (x_n) \leq 0$. However, because $x_n < m$, we know $f^{(k)} (x_n) \geq 0$ for all $k < n$. By \cref{diffineq-fineq}, we have 
    \[ 0 \geq f^{(n)} (x_n) > \sum_{k = 0}^{n-1} c_k (x_n) f^{(k)} (x_n) \geq 0.\]
    Of course, it is impossible for $0 > 0$, so $f(x) = v(x) - y(x) \geq 0$ on $[a,b)$. This implies the lemma.
\end{proof}

\begin{proposition}\label{negdiffineq}
    Let $y, v: (a,b] \rightarrow \RR$ be differentiable functions such that $y(b) = v(b)$ and for all $x \in (a,b]$,
    \[ -y'(x) < C(x) + c_0(x) y, \]
    \[ -v'(x) = C(x) + c_0(x) v, \]
    where $c_0(x) : (a,b] \rightarrow \RR^{\geq 0}$ and $C : (a,b] \rightarrow \RR$. Then, $y \leq v$ on $(a,b)$.
\end{proposition}

\begin{proof}
    Let $y_\star (x) = y(-x)$ and $v_\star (x) = v(-x)$. Using our given conditions, the substitution $x = -z$ yields
    \[ y_\star'(z) < C(-z) + c_0(-z) y_\star(-z), \]
    \[ v_\star'(z) = C(-z) + c_0(-z) y_\star(-z), \]
    for all $z \in [-b,-a)$. This now satisfies \cref{diffineq}, which implies that $y_\star (z) \leq v_\star(z)$ on $[-b,-a)$, and thus that $y(x) \leq v(x)$ on $(a,b]$.
\end{proof}

\begin{proposition}\label{limitfromupperlower}
    Consider function $f : (0, a) \rightarrow \RR$. If for all $\eps > 0$, there exists $\ell = \ell(\eps) \in \RR$ and $\delta = \delta(\eps) \in \RR^{>0}$ such that for all $x \in (0 , \delta)$, we have $f(x) \in (\ell, \ell + \eps)$, then $\lim_{x \rightarrow 0^+} f(x)$ exists.
\end{proposition}

\begin{proof} 
    The proof can be done by showing $\lim_{x \rightarrow 0^+} f(x) = \sup_{ \eps > 0 } l(\eps)$ and is omitted here for brevity.
\end{proof}

\section{Warmup: The Multidisperse Ghost Process}\label{ghost-m}

Here, we compute the behavior of the expected saturation density of multidisperse ghost processes, formally defined below. This process admits a simpler analysis than the multidisperse process because each step is independent of the previous ones. While in the usual multidisperse process, we must consider if any previous segments have been parked, we do not have to consider this in the ghost process.

\begin{definition}[Multidisperse Ghost Process]\label{ghost-m-defn}
    Fix lengths $\ell_1 = 1, \ldots, \ell_n \in \RR^{>0}$ and rates $q_1, \ldots, q_n \in \RR^{>0}$ such that $\ell_1 < \ell_2 < \cdots < \ell_n$ and $q_1 + \cdots + q_n = 1.$ Then, let the $( (\ell_1,q_1), \cdots, (\ell_n,q_n) )$-\emph{multidisperse ghost process} (abbreviated as $\mcMG( (\ell_1,q_1), \cdots, (\ell_n,q_n) )$) be the stochastic process defined as follows. 

    Initialize:
    \begin{itemize}
        \item $I_0 = [0, L],$ the empty interval not occupied by parked or ghost segments,
        \item $P_0 = \emptyset $, the set of parked segments.
    \end{itemize}

    At times $T = 1, 2, \ldots$,
    \begin{enumerate}
        \item Sample candidate segment center $c \xleftarrow{R} \left[- \frac{1}{2} \ell_n,L+ \frac{1}{2} \ell_n\right]$, and sample $k \in [n]$ with probability $q_k$. We say that segment $\left(c - \frac{1}{2} \ell_k,c+\frac{1}{2} \ell_k\right)$ is a \emph{candidate segment} of \emph{type $k$}.

        \item If $\left(c - \frac{1}{2} \ell_k,c+\frac{1}{2} \ell_k\right) \subseteq I_{T-1}$, let $P_T = P_{T-1} \cup \{ \left(c - \frac{1}{2} \ell_k,c+\frac{1}{2} \ell_k\right) \}.$ We say the candidate segment has been \emph{parked}. Otherwise, let $P_T = P_{T-1}$, and we say that the candidate segment has become a \emph{ghost}.

        \item Regardless of whether the segment has been parked, let $I_T = I_{T-1} \setminus \left(c - \frac{1}{2} \ell_k,c+\frac{1}{2} \ell_k\right)$.
    \end{enumerate}

    We define the \emph{type $k$ saturation number} as the total number of type $k$ segments parked, at saturation:
    \begin{equation}\label{ghost-N-defn}
        N_{k,L} = \lim_{t\rightarrow \infty} | \{ A \in P_t : \lambda(A) = \ell_k \} |.
    \end{equation}
    Moreover, let $J_L$ be the random variable representing the total length of parked segments, viz.
    \begin{equation}\label{ghost-J-defn}
        J_L = \lim_{T \rightarrow \infty} \sum_{A \in P_T} \lambda(A). 
    \end{equation}
    This can also be expressed in terms of the variables $N_{k,L}$, as $ J_L = \sum_{k=1}^n \ell_k N_{k,L}. $
\end{definition}

\begin{remark}
    This definition allows segment centers to be parked on $[-\frac{1}{2} \ell_n, 0]$ and $[L, L + \frac{1}{2} \ell_n]$. If this weren't allowed, then the behavior of the process would change at the ends of the interval $[0,L]$. Moreover, we use segment centers in this definition to simplify the following proofs, but for most other RSA processes, (c.f. \cref{RSA-defn}, \cref{multidisperse-defn}), we use the left endpoints of parked segments to reference them.
\end{remark}

For the remainder of the section, the $\ell_i$'s will always be positive reals representing the segment lengths in the multidisperse ghost process, and the $q_i$'s will always be positive reals summing to $1$ that represent the probabilities of choosing each segment. We implicitly assume $\ell_1 = 1$ and $\ell_1 < \ell_2 < \cdots < \ell_n$. Moreover, we will always use $\overline{\ell}$ to denote the \emph{average segment length}, defined as follows:
\begin{definition}
    In $\mathcal{M}_G ((\ell_1,q_1), \cdots, (\ell_n,q_n))$, the \emph{average segment length} $\overline{\ell}$ is given by $\overline{\ell} = \sum_{i = 1}^n q_i \ell_i.$
\end{definition}

We are now able to derive a simple asymptotic for the multidisperse ghost process.

\begin{proposition}\label{ghost-m-lim}
    In $\mcMG ((\ell_1,q_1), \cdots,(\ell_n,q_n))$, the type $k$ saturation number $N_{k,L}$ (c.f. \cref{ghost-m-defn}) satisfies
    \[ \lim_{L \rightarrow \infty } \frac{\EE[ N_{k,L} ]}{L} = \frac{q_k}{\overline{\ell} + \ell_k }.\]

    Moreover,
    \begin{equation}\label{ghost-m-JL}
        \lim_{L \rightarrow \infty } \frac{\EE[ J_{L} ]}{L} = \sum_{k=1}^n \frac{q_k \ell_k}{\overline{\ell} + \ell_k},
    \end{equation}
\end{proposition}

\begin{proof}
    Fix $k \in [n]$. At time $T$, we choose a type $k$ candidate segment with probability $q_k$, and this segment is entirely contained within $[0,L]$ exactly when the segment center $c \in \left[\frac{1}{2} \ell_k,L - \frac{1}{2} \ell_k\right]$, which occurs with probability $\frac{L-\ell_k}{L+\ell_n}$.

    Now, assume we have chosen candidate segment $\left(c - \frac{1}{2} \ell_k,c+\frac{1}{2} \ell_k\right) \subseteq [0,L],$ so $c$ is fixed. Fix time $T' < T$. To compute the probability that the candidate segments at times $T$ and $T'$ intersect, we condition over the type of the candidate segment at time $T'$.
    
    Assume that at time $T'$, we have chosen a candidate segment of type $i$ with center $c'$. The candidate segments at times $T$ and $T'$ do not intersect when $|c-c'| > \frac{1}{2} ( \ell_i + \ell_k )$. Since $c'$ is uniformly distributed on $\left[- \frac{1}{2} \ell_n,L+ \frac{1}{2} \ell_n\right]$, the candidate segments do not intersect with probability $ \frac{L+\ell_n- (\ell_i + \ell_k)}{L + \ell_n} .$

    The probability of choosing type $i$ at time $T'$ is simply $q_i$, so the probability that the candidate segments at times $T$ and $T'$ do not intersect is
    \[ \sum_{i=1}^n q_i \cdot \frac{L+\ell_n- (\ell_i + \ell_k)}{L + \ell_n} = 1 - \frac{\overline{\ell} + \ell_k }{L+\ell_n}. \]

    The candidate segments before time $T$ are independently chosen, so in fact, the probability that the candidate at time $T$ intersects with no previous candidate is $\left( 1 - \frac{\overline{\ell} + \ell_k }{L+\ell_n} \right)^{T-1}.$ Thus, the probability that at time $T$, we successfully park a type $k$ segment is
    \[ q_k \cdot \frac{L-\ell_k}{L+\ell_n} \cdot \left( 1 - \frac{\overline{\ell} + \ell_k }{L+\ell_n} \right)^{T-1}. \]
    Summing over times $T$ from $1$ to $\infty$, the expected number of parked type $k$ segments is
    \[ \EE[N_{k,L} ] = \sum_{T=1}^\infty q_k \cdot \frac{L-\ell_k}{L+\ell_n} \cdot \left( 1 - \frac{\overline{\ell} + \ell_k}{L+\ell_n} \right)^{T-1} = \frac{q_k (L - \ell_k)}{\overline{\ell} + \ell_k}. \]
    Thus, $\lim_{L \rightarrow \infty} \frac{\EE[N_{k,L} ]}{L} = \frac{q_k}{\overline{\ell} + \ell_k}$. The formula for $J_L$ follows by $J_L = \sum_{k=1}^n \ell_k N_{k,L}$.
\end{proof}

We now prove bounds on $\lim_{L \rightarrow \infty}\frac{\EE[J_L]}{L}$.

\begin{corollary}\label{ghost-m-bound2}
    In $\mcMG ((\ell_1,q_1), (\ell_2,q_2))$, the jamming length satisfies
    \[ \frac{2 \sqrt{\ell_1 \ell_2} }{(\sqrt{\ell_1} + \sqrt{\ell_2})^2 } \leq \lim_{L \rightarrow \infty } \frac{\EE[ J_{L} ]}{L} \leq \frac{1}{2}. \]
    The maximum is achieved when either $q_1 = 0$ or $q_2 = 0$. The minimum is achieved when $ \overline{\ell} = \sqrt{\ell_1 \ell_2}. $ 
\end{corollary}

\begin{proof}
    Fix the lengths $\ell_1 < \ell_2$, and let $F : [0,1] \rightarrow \RR$ be given as $\lim_{L \rightarrow \infty} \frac{\EE[J_L]}{L}$ when $q_1 = q$ and $q_2 = 1-q$, viz.
    \[ F(q) = \frac{q \ell_1}{\overline{\ell} + \ell_1} + \frac{(1-q) \ell_2} {\overline{\ell} + \ell_2}. \]
    We wish to bound $F$ both above and below. Taking the derivative and simplifying, we have
    \[ F'(q) = (\ell_1 + \ell_2) \left( \frac{\ell_1}{(\ell_1 + \overline{\ell})^2} - \frac{\ell_2}{(\ell_2 + \overline{\ell})^2} \right). \]
    When $ q < \frac{\sqrt{\ell_2}}{\sqrt{\ell_1} + \sqrt{\ell_2}} $, then $\overline{\ell} > \sqrt{\ell_1 \ell_2}$, and it can be shown that $F'(q) < 0$. Meanwhile, when $q > \frac{\sqrt{\ell_2}}{\sqrt{\ell_1} + \sqrt{\ell_2}}$, then $\overline{\ell} < \sqrt{\ell_1 \ell_2}$ and $F'(q) > 0$. Thus, $F$ must have a global minimum at $q = \frac{\sqrt{\ell_2}}{\sqrt{\ell_1} + \sqrt{\ell_2}}$, where $F (q) = \frac{2 \sqrt{\ell_1 \ell_2} }{(\sqrt{\ell_1} + \sqrt{\ell_2})^2 }$ and $\overline{\ell} = \sqrt{\ell_1 \ell_2}$. Moreover, the global maximum must lie at either $q=0$ or $q=1$. Upon inspection, both yield a maximum of $F(q) = \frac{1}{2}$. This proves the lemma.
\end{proof}

We now extend our work with the process $\mcMG ((\ell_1,q_1), (\ell_2,q_2))$ to the more general process $\mcMG ((\ell_1,q_1), \ldots, (\ell_n,q_n))$.

\begin{corollary}\label{ghost-m-boundn}
    In $\mcMG ((\ell_1,q_1), \cdots, (\ell_n,q_n))$, the jamming length satisfies
    \[ \frac{2 \sqrt{\ell_1 \ell_n} }{(\sqrt{\ell_1} + \sqrt{\ell_n})^2 } \leq \lim_{L \rightarrow \infty } \frac{\EE[ J_{L} ]}{L} \leq \frac{1}{2}. \] 
\end{corollary}

\begin{proof}
    Fix the lengths $\ell_1 < \cdots < \ell_n$. We wish to bound the function $F : \RR^n \rightarrow \RR$, given by 
    \[ F(q_1, \ldots, q_n) =  \lim_{L \rightarrow \infty} \frac{\EE[J_L]}{L} = \sum_{k=1}^n \frac{q_k \ell_k}{\overline{\ell} + \ell_k}. \]
    Let $\mcP$ be the region $ \{ (q_1, \ldots, q_n) \in \RR^n : q_1 + \cdots + q_n = 1, q_1, \ldots, q_n \geq 0 \} .$ We only consider $F$ in this region $\mcP$, and we will induct on the value of $n$, where the base case with $n=2$ is proved in \cref{ghost-m-bound2} (and the case where $n=1$ is trivial).
    
    Assume now that \cref{ghost-m-boundn} holds for $n = M-1$. We will show it holds for $n = M$, where $M \geq 3$. We first investigate relative extrema within the region $R$ using Lagrange multipliers. If such an extrema existed, we must have $\nabla F = c \langle 1, \cdots, 1 \rangle$ for some real $c$, or equivalently, $\dddf{F}{q_1} = \cdots = \dddf{F}{q_M} .$ 

    Fix arbitrary $i, j \in [M]$. Then, $\dddf{F}{q_i} = \dddf{F}{q_j}$ becomes
    \[ \ell_i \left( \frac{1}{\overline{\ell} + \ell_i} - \sum_{k \in [M]} \frac{q_k \ell_k }{ (\overline{\ell} + \ell_k)^2} \right) = \ell_j \left( \frac{1}{\overline{\ell} + \ell_j} - \sum_{k \in [M]} \frac{q_k \ell_k }{ (\overline{\ell} + \ell_k)^2} \right),\] 
    which simplifies to
    \[ \frac{\overline{\ell}}{(\overline{\ell} + \ell_i)(\overline{\ell} + \ell_j)} = \sum_{k \in [M]} \frac{q_k \ell_k }{ (\overline{\ell} + \ell_k)^2}. \]
    The right hand side is a constant expression not depending on $i$ and $j$. Thus, $\frac{\overline{\ell}}{(\overline{\ell} + \ell_i)(\overline{\ell} + \ell_j)}$ too must be constant for any choice of $i$ and $j$, which only can occur when all the lengths are equal, viz. $\ell_1 = \cdots = \ell_M$. This cannot happen, as our lengths are all distinct. Thus, $F$ has no relative extrema in the interior of the region $\mcP$.

    Both the minimum and maximum of $F$ must then lie on the boundary of $R$, so fix an arbitrary $q_i = 0$. This now reduces our problem to the case with $n = M-1$, and by the inductive hypothesis, the maximum is always $\frac{1}{2}$.
    
    Meanwhile, if $i = 1$, the minimum is $\frac{2 \sqrt{\ell_2 \ell_M} }{(\sqrt{\ell_2} + \sqrt{\ell_M})^2 }$. If $i = M$, the minimum is $\frac{2 \sqrt{\ell_1 \ell_{M-1}} }{(\sqrt{\ell_1} + \sqrt{\ell_{M-1}})^2 }$. Otherwise, the minimum is $\frac{2 \sqrt{\ell_1 \ell_M} }{(\sqrt{\ell_1} + \sqrt{\ell_M})^2 }$. By writing  $\frac{2 \sqrt{\ell_1 \ell_{M}} }{(\sqrt{\ell_1} + \sqrt{\ell_{M}})^2 }$ as $\frac{1}{2} - \left(\frac{\sqrt{\ell_1} - \sqrt{\ell_M}}{\sqrt{\ell_1} + \sqrt{\ell_M}} \right)^2$ and the others analogously, it is straightforward to show that $\frac{2 \sqrt{\ell_1 \ell_M} }{(\sqrt{\ell_1} + \sqrt{\ell_M})^2 }$ is the least of the three minima, which completes the proof.
\end{proof}

\begin{remark}
    In the multidisperse ghost process, we are able to completely bound the jamming density $\lim_{L \rightarrow \infty} \frac{\EE[J_L]}{L}$. We are generally unable to do this in the multidisperse process, which admits a much more complex analysis.
\end{remark}

\section{The Multidisperse RSA Process}\label{multi}

In this section, we study multidisperse RSA processes that we formally define below.
\begin{definition}[Multidisperse Process]\label{multidisperse-defn}
    Fix segment length $1 = \ell_1 < \cdots < \ell_n$ and probabilities $q_1, \ldots, q_n \in \RR^{> 0}$ (with $\sum_{i \in [n]} q_i = 1$). The \textit{$((\ell_1,q_1), \cdots, (\ell_n,q_n))$-multidisperse process}, abbreviated as $\mcM( (\ell_1,q_1), \cdots, (\ell_n,q_n) )$ is defined to be the following random process:

    Initialize:
    \begin{itemize}
        \item $I_0 = [0, L],$ the empty region not occupied by parked segments,
        \item $P_0 = \emptyset $, the set of parked segments.
    \end{itemize}
    
    Then, for $T = 1, 2, \ldots:$
    \begin{enumerate}
        \item Sample left endpoint $b \xleftarrow{R} [0,L]$, and segment type $i \in [n]$ with $\P(i = k) = q_k$. We say that segment $(b,b+\ell_i)$ is a \emph{candidate segment} of \emph{type} $i$.
        \item If $(b,b+\ell_i) \subseteq I_{T-1}$, let $I_T = I_{T-1} \setminus (b,b+\ell_i)$ and $P_T = P_{T-1} \cup \{ (b,b+\ell_i) \}.$ We say the segment $(b,b+\ell_i)$ has been \emph{parked}. Otherwise, let $I_T = I_{T-1}$ and $P_T = P_{T-1}$, and we say that the segment $(b,b+\ell_i)$ has been \emph{rejected}.
    \end{enumerate}

    We define the \emph{type $k$ saturation number} to be the total number of type $k$ segments parked at saturation:
    \[ N_{k,L} := \lim_{T\rightarrow \infty} | \{ A \in P_T : \lambda(A) = \ell_k\} |,\]
    where $\lambda( A )$ denotes the Lesbegue measure of $A$.
\end{definition}

Like in the multidisperse ghost process, for the remainder of the section, the $\ell_i$'s will always be positive reals representing the segment lengths, and the $q_i$'s will always be positive reals summing to $1$ that represent the probabilities of choosing each segment. The shortest segment length, $\ell_1$, is always implicitly taken to be $1$. Moreover, we will let $\ol \ell$ be the \textit{average segment length}, so that $$\overline{\ell} := \sum_{i \in [n]} q_i \ell_i.$$  

We will study the expected number of type $k$ segments placed. We begin with the following integral recurrence formula:

\begin{proposition}\label{multi-recurrence}
    In $\mathcal{M} ((\ell_1,q_1), \cdots, (\ell_n,q_n))$, for $L \geq \ell_n$,
    \[ \EE[N_{k,L}] = \frac{1}{L-\overline{\ell}}\left( q_k(L-\ell_k) + 2\sum_{i=1}^n q_i\int_0^{L-\ell_i} \EE[N_{k,s}] \df{s} \right) . \]
\end{proposition}

\begin{proof}
Consider parking segments on interval $[0, L]$ with $L \ge \ell_n$. At time $T = 1$, note that we succeed in parking a segment of type $i$ with probability $q_i \cdot \frac{L - \ell_i}{L}$, as we must first choose a type $i$ segment and then successfully park it by choosing its left endpoint to be in $[0,L-\ell_i]$, which occurs with probability $\frac{L-\ell_i}{L}$. These two events are independent.

Let $A_i$ be the event that the first segment eventually parked is type $i$. If no segment is parked on the first attempt, we repeatedly make new, independent attempts, so $\P(A_i) \propto q_i \cdot \frac{L-\ell_i}{L}.$ Since $\sum_{i \in [n]} \P(A_i) = 1$, by renormalizing the probabilities we see that $\PP(A_i) = \frac{q_i (L-\ell_i)}{L-\overline{\ell}}. $

Conditioned on $A_i$, let the position of the first segment's left endpoint be $s$. Note that $s$ is a random variable uniformly distributed on $[0,L-\ell_i]$. After placing the length $\ell_i$ segment down on $[0,L]$ at left endpoint $s$, $[0, L]$ is then broken into two subintervals of length $s$ and $L-\ell_i-s$, on which the same multidisperse process continues. The expected numbers of type $k$ intervals placed on the two subintervals are then $\EE[N_{k,s}]$ and $\EE[N_{k,L-\ell_i-s}]$, respectively. Thus, integrating over the random variable $s$, we have that when $i \neq k$,
    \[ \EE[N_{k,L} \mid A_i] = \frac{1}{L-\ell_i} \int_0^{L-\ell_i} ( \EE[N_{k,s}] +\EE[N_{k,L-\ell_i-s}] ) \df{s} = \frac{2}{L-\ell_i} \int_0^{L-\ell_i} \EE[N_{k,s}] \df{s}.\]
    When $i = k$, we have just placed a type $k$ segment and must accordingly add $1$ to $\EE[N_{k,L}]$:
    \[ \EE[N_{k,L} \mid A_k] = 1 + \frac{2}{L-\ell_k} \int_0^{L-\ell_k} \EE[N_{k,s}] \df{s}. \]
    Combining these gives
    \[ \EE[N_{k,L}] = \sum_{i = 1}^n \PP[A_i] \EE[N_{k,L} \mid A_i] = \frac{1}{L-\overline{\ell}}\left( q_k(L-\ell_k) + 2\sum_{i=1}^n q_i\int_0^{L-\ell_i} \EE[N_{k,s}] \df{s} \right) . \] 
\end{proof}

Using this recurrence, we will use Laplace transforms to derive precise asymptotics for $\EE[N_{k,L}]$ as $L \rightarrow \infty$. We will make use of $\Ein: \RR \rightarrow \RR$ to denote the modified exponential integral function, viz.
\begin{equation}\label{ein-defn}
    \Ein(z) = \int_0^z \frac{1-e^{-t}}{t} \df{t}.
\end{equation}

\begin{remark}\label{bidisperse-remark}
    The \emph{bidisperse process}, i.e. the multidisperse process with $n=2$ different segment lengths, has been investigated before by various authors. In particular, Subashiev and Luryi derive in \cite{subashiev2007} an exact expression for $ \lim_{L \rightarrow \infty} \frac{\EE[N_{k,L}]}{L} $ in the bidisperse process. Our work here extends their work to the multidisperse process, and by choosing $n=2$, one recovers their formula.
\end{remark}

\begin{theorem}\label{multi-constant}
    Consider $\mcM((\ell_1,q_1), \cdots, (\ell_n,q_n))$, and fix $k \in [n]$. Define functions $P_{i;k} : \RR^{>0} \rightarrow \RR$ as
    \begin{equation}\label{multi-Pik-defn}
        P_{i;k}(s) := \int_{\ell_n - \ell_i}^{\ell_n} \EE[N_{k,L}] e^{-sL} \df{L},
    \end{equation}
    and define $G_k : \RR^{>0} \rightarrow \RR$ as
    \begin{equation}\label{multi-Gk-defn}
        G_k(s) := e^{- (\ell_n - \overline{\ell}) s} \Big( q_k + s (\ell_n - \overline{\ell})  \EE[N_{k,\ell_n}] \Big) + 2 s e^{\overline{\ell} s} \sum_{i = 1}^n q_i e^{-\ell_i s} P_{i;k} (s).
    \end{equation}
    Then,
    \begin{equation}\label{multi-constant-eqn}
        \lim_{L \rightarrow \infty } \frac{\EE[ N_{k,L} ]}{L} = \int_0^\infty G_k(t) \exp\left(-2 \sum_{i=1}^n q_i \Ein (\ell_i t)\right) \df{t}.
    \end{equation}
\end{theorem}

\begin{remark}
    The functions $P_{i;k}$ arise from the behavior of $ N_{k,L} $ when $L \leq \ell_n$, which isn't governed by the integral recurrence in \cref{multi-recurrence}.
\end{remark}

\begin{example}
    Consider $\EE[N_{1,L}]$ in the multidisperse process $\mcM ( (1,\frac{1}{2}), (\frac{13}{10}, \frac{3}{10}), (\frac{3}{2}, \frac{1}{5}) )$, which has 3 possible segment lengths of $1$, $\frac{13}{10}$, and $\frac{3}{2}$. Note that because $\EE[N_{1,L}] = 0$ when $L < 1$, each $P_{i;1}$ is equal to $\int_1^{3/2} \EE[N_{1,L}]e^{-sL} \df{L}$ in this case. To compute $P_{i;1}$, we compute that $\EE[N_{1,L}]$ is $1$ when $1 \leq L < \frac{13}{10}$ and $ \frac{\frac{1}{2} \left(L-1\right)}{\frac{1}{2}(L-1) + \frac{3}{10}\left(L-\frac{13}{10}\right) } $ when $\frac{13}{10} \leq L < \frac{15}{10}$. Then, noting $\overline{\ell} = \frac{31}{100}$, we have
    \[ G_1 (s) := e^{-31s/100} \left(\frac{1}{2} + \frac{1}{4} s\right) + 2 s \left( \frac{1}{2} e^{19s/100} + .3 e^{-11s/100} + .2 e^{-31s/100} \right) \int_1^{3/2} \EE[N_{1,L}]e^{-sL} \df{L}, \]
    and numerically integrating \cref{multi-constant-eqn} yields
    \[ \lim_{L \rightarrow \infty} \frac{\EE[ N_{1,L} ]}{L} = \int_0^\infty G_1(t) \exp\left( -2 \left( \frac{1}{2} \Ein(t) + \frac{3}{10} \Ein \left(\frac{13}{10}t\right) + \frac{1}{5} \Ein\left(\frac{3}{2} t\right) \right) \right) \df{t} \approx .4204. \]
    That is, $\EE[N_{1,L}] \sim .4204 L$. We may similarly compute that $\EE[N_{2,L}] \sim .1655L$ and $\EE[N_{3,L}] \sim 0.0949 L$ (c.f. \cref{multi-example-fig}). With these values, we see that the total length covered by all segments grows asymptotically equal to $.7778 L$.
\end{example}

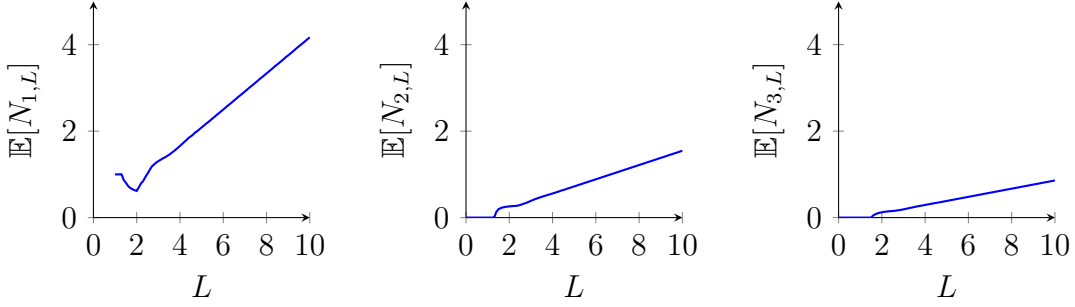
\begin{figure}
    \centering
    \begin{subfigure}{.3\textwidth}
        \begin{tikzpicture}[scale=1]
            \begin{axis} [axis lines = left,
                xlabel = \(L\),
                ylabel = {$\EE[N_{1,L}]$},
                width=.9\textwidth,
                height=.9\textwidth,
                ymin=0,
                xmin=0,
                xmax=10,
                ymax=5
            ]
            % \addplot [color=black, thick, dotted] coordinates {
            %     (0,-0.264415093126)(10,7.51360933438)
            % };
            \addplot [color=blue, thick] coordinates { 
                (1,1)(1.1,1)(1.2,1)(1.3,1)(1.4,0.869565)(1.5,0.806452)(1.6,0.731707)(1.7,0.686275)(1.8,0.655738)(1.9,0.633803)(2,0.617284)(2.1,0.703297)(2.2,0.792079)(2.3,0.846847)(2.4,0.939848)(2.5,1.01605)(2.6,1.1021)(2.7,1.1834)(2.8,1.22876)(2.9,1.27268)(3,1.30662)(3.1,1.33575)(3.2,1.36779)(3.3,1.39273)(3.4,1.42194)(3.5,1.4546)(3.6,1.49199)(3.7,1.53108)(3.8,1.57049)(3.9,1.61358)(4,1.65816)(4.1,1.703)(4.2,1.75042)(4.3,1.79246)(4.4,1.84477)(4.5,1.8794)(4.6,1.9176)(4.7,1.96787)(4.8,2.00106)(4.9,2.04233)(5,2.08591)(5.1,2.12466)(5.2,2.16985)(5.3,2.20727)(5.4,2.24875)(5.5,2.29038)(5.6,2.33209)(5.7,2.37632)(5.8,2.41575)(5.9,2.4576)(6,2.4995)(6.1,2.54135)(6.2,2.58563)(6.3,2.62501)(6.4,2.66678)(6.5,2.70856)(6.6,2.7503)(6.7,2.79448)(6.8,2.83375)(6.9,2.87547)(7,2.91719)(7.1,2.9589)(7.2,3.00308)(7.3,3.04235)(7.4,3.08408)(7.5,3.12582)(7.6,3.16755)(7.7,3.21175)(7.8,3.25102)(7.9,3.29276)(8,3.3345)(8.1,3.37623)(8.2,3.41797)(8.3,3.45971)(8.4,3.50551)(8.5,3.54318)(8.6,3.58491)(8.7,3.62665)(8.8,3.66839)(8.9,3.71419)(9,3.75186)(9.1,3.79361)(9.2,3.83535)(9.3,3.87711)(9.4,3.92295)(9.5,3.96311)(9.6,4.00487)(9.7,4.04417)(9.8,4.08594)(9.9,4.13018)(10,4.17195)
            };
            \end{axis}
        \end{tikzpicture}
    \end{subfigure}%
    \begin{subfigure}{.3\textwidth}
        \begin{tikzpicture}[scale=1]
            \begin{axis} [axis lines = left,
                xlabel = \(L\),
                ylabel = {$\EE[N_{2,L}]$},
                width=.9\textwidth,
                height=.9\textwidth,
                ymin=0,
                xmin=0,
                xmax=10,
                ymax=5
            ]
            % \addplot [color=black, thick, dotted] coordinates {
            %     (0,-0.264415093126)(10,7.51360933438)
            % };
            \addplot [color=blue, thick] coordinates { 
                (0,0)(0.1,0)(0.2,0)(0.3,0)(0.4,0)(0.5,0)(0.6,0)(0.7,0)(0.8,0)(0.9,0)(1,0)(1.1,0)(1.2,0)(1.3,0)(1.4,0.130435)(1.5,0.193548)(1.6,0.219512)(1.7,0.235294)(1.8,0.245902)(1.9,0.253521)(2,0.259259)(2.1,0.263736)(2.2,0.267327)(2.3,0.27027)(2.4,0.278334)(2.5,0.292349)(2.6,0.307532)(2.7,0.32685)(2.8,0.3442)(2.9,0.365037)(3,0.38671)(3.1,0.407728)(3.2,0.428767)(3.3,0.446075)(3.4,0.463433)(3.5,0.480024)(3.6,0.49667)(3.7,0.512166)(3.8,0.526866)(3.9,0.542191)(4,0.557734)(4.1,0.573496)(4.2,0.590223)(4.3,0.605766)(4.4,0.624867)(4.5,0.638818)(4.6,0.653831)(4.7,0.673295)(4.8,0.687287)(4.9,0.703983)(5,0.721446)(5.1,0.737245)(5.2,0.755134)(5.3,0.770343)(5.4,0.786839)(5.5,0.803331)(5.6,0.819802)(5.7,0.837105)(5.8,0.852765)(5.9,0.869242)(6,0.885743)(6.1,0.902237)(6.2,0.91958)(6.3,0.935248)(6.4,0.951749)(6.5,0.968258)(6.6,0.984759)(6.7,1.00212)(6.8,1.01776)(6.9,1.03426)(7,1.05076)(7.1,1.06725)(7.2,1.08462)(7.3,1.10024)(7.4,1.11674)(7.5,1.13324)(7.6,1.14973)(7.7,1.1671)(7.8,1.18272)(7.9,1.19922)(8,1.21572)(8.1,1.23221)(8.2,1.24871)(8.3,1.26521)(8.4,1.28317)(8.5,1.2982)(8.6,1.3147)(8.7,1.3312)(8.8,1.34769)(8.9,1.36567)(9,1.38069)(9.1,1.39719)(9.2,1.41369)(9.3,1.43019)(9.4,1.44819)(9.5,1.46411)(9.6,1.48062)(9.7,1.49622)(9.8,1.51273)(9.9,1.53015)(10,1.54666)
            };
            \end{axis}
        \end{tikzpicture}
    \end{subfigure}%
    \begin{subfigure}{.3\textwidth}
        \begin{tikzpicture}[scale=1]
            \begin{axis} [axis lines = left,
                xlabel = \(L\),
                ylabel = {$\EE[N_{3,L}]$},
                width=.9\textwidth,
                height=.9\textwidth,
                ymin=0,
                xmin=0,
                xmax=10,
                ymax=5
            ]
            % \addplot [color=black, thick, dotted] coordinates {
            %     (0,-0.264415093126)(10,7.51360933438)
            % };
            \addplot [color=blue, thick] coordinates { 
                (0,0)(0.1,0)(0.2,0)(0.3,0)(0.4,0)(0.5,0)(0.6,0)(0.7,0)(0.8,0)(0.9,0)(1,0)(1.1,0)(1.2,0)(1.3,0)(1.4,0)(1.5,0)(1.6,0.0487805)(1.7,0.0784314)(1.8,0.0983607)(1.9,0.112676)(2,0.123457)(2.1,0.131868)(2.2,0.138614)(2.3,0.144144)(2.4,0.14876)(2.5,0.152672)(2.6,0.157744)(2.7,0.165247)(2.8,0.172383)(2.9,0.180997)(3,0.190648)(3.1,0.201079)(3.2,0.212631)(3.3,0.223272)(3.4,0.2342)(3.5,0.244784)(3.6,0.255334)(3.7,0.265197)(3.8,0.274431)(3.9,0.283768)(4,0.292972)(4.1,0.302084)(4.2,0.311549)(4.3,0.320347)(4.4,0.330889)(4.5,0.338852)(4.6,0.347272)(4.7,0.358088)(4.8,0.366165)(4.9,0.375671)(5,0.385613)(5.1,0.394723)(5.2,0.404937)(5.3,0.413765)(5.4,0.423265)(5.5,0.432759)(5.6,0.442236)(5.7,0.452144)(5.8,0.461175)(5.9,0.470632)(6,0.480096)(6.1,0.489552)(6.2,0.49946)(6.3,0.508473)(6.4,0.517931)(6.5,0.527394)(6.6,0.536854)(6.7,0.546777)(6.8,0.555778)(6.9,0.565238)(7,0.574701)(7.1,0.584161)(7.2,0.594093)(7.3,0.603085)(7.4,0.612546)(7.5,0.622008)(7.6,0.631468)(7.7,0.641408)(7.8,0.650391)(7.9,0.659852)(8,0.669315)(8.1,0.678775)(8.2,0.688236)(8.3,0.697698)(8.4,0.70796)(8.5,0.716621)(8.6,0.726082)(8.7,0.735543)(8.8,0.745005)(8.9,0.755277)(9,0.763929)(9.1,0.773392)(9.2,0.782856)(9.3,0.792321)(9.4,0.802607)(9.5,0.811754)(9.6,0.821221)(9.7,0.83019)(9.8,0.839658)(9.9,0.849627)(10,0.859097)
            };
            \end{axis}
        \end{tikzpicture}
    \end{subfigure}
    \caption{Plots of $\EE[N_{k,L}]$ in $\mcM ( (1,.5), (1.3, .3), (1.5, .2) )$.}
    \label{multi-example-fig}
\end{figure}

\begin{proof}[Proof of \cref{multi-constant}]
    Fix $k \in [n]$. We first derive a formula for the Laplace transform of $\EE[N_{k,L}]$ and then determine the behavior of this Laplace transform around $0$. Finally, we apply~\cref{Tauberian} to determine the behavior of $\EE[N_{k,L}]$ as $L \rightarrow \infty$.

Define 
    \begin{equation}
        \varphi(s) := \int_{\ell_n}^\infty \EE[N_{k,L}] e^{-sL} \df{L}.
    \end{equation}
    Note $\varphi(s)$ is not exactly the Laplace transform of $\EE[N_{k,L}]$. Notably, the lower limit of $L$ is $\ell_n$ rather than $0$, as the multidisperse process has a fundamentally different behavior for small $L$ (segments of length $\ell_n$ are never parked). 
    
    For brevity, define constants
    \begin{equation}\label{multi-rho-defn}
        \rho := \ell_n - \overline{\ell} \qquad \text{and} \qquad \rho_i := \ell_n - \ell_i
    \end{equation}
    for $i \in [n]$. The $\rho_i$'s are the differences in length between the largest segment and the other segments, and $\rho$ is the difference between the largest segment length and the mean. Because $\ell_n$ is the largest segment length, $\rho$ and $\rho_i$ are all nonnegative.

    We now proceed to formulate a differential equation for $\varphi(s)$:
    \begin{lemma}\label{multi-laplace-diffeq}
        Let $w(s) = e^{\overline{\ell} s} \varphi(s)$. Then, with $G_k(s)$ as defined in \cref{multi-constant},
        \[
            w'(s) + \frac{2 w(s)}{s} \sum_{i=1}^n q_i e^{-\ell_i s} + \frac{G_k(s)}{s^2} = 0.
        \]
    \end{lemma}
    \begin{proof}
        By \cref{multi-recurrence} with $L+\ell_n$ as the interval length, we have \linechecktext{\[ \EE[N_{k,L+\ell_n}] = \frac{1}{L+\ell_n-\overline{\ell}}\left( q_k(L+\ell_n-\ell_k) + 2\sum_{i=1}^n q_i\int_0^{L+\ell_n-\ell_i} \EE[N_{k,t}] \df{t} \right). \]}
        \[ \EE[N_{k,L+\ell_n}] = \frac{1}{L+\rho}\left( q_k(L+\rho_k) + 2\sum_{i=1}^n q_i\int_0^{L+\rho_i} \EE[N_{k,t}] \df{t} \right). \]

        This equation holds for all $L \geq 0$. \linechecktext{Multiply both sides by $L+\rho$:
        \[ (L + \rho) \EE[N_{k,L+\ell_n}] = q_k(L + \rho_k) + 2 \sum_{i = 1}^n q_i \int_0^{L + \rho_i} \EE[N_{k,t}] \df{t}. \]} Rearrange and differentiate with respect to $L$ to get
        \[ \ddf{L} \Big[ (L + \rho) \EE[N_{k,L+\ell_n}] \Big] = q_k + 2\sum_{i=1}^n q_i \EE[N_{k,L+\rho_i}],\]
        which has Laplace transform
        \begin{equation}\label{multi-laplace}
            \int_0^\infty \ddf{L} \Big[ (L + \rho) \EE[N_{k,L + \ell_n}] \Big] e^{-sL} \df{L} = q_k \int_0^\infty e^{-sL} \df{L} + 2\sum_{i=1}^n q_i \int_0^\infty  \EE[N_{k,L+\rho_i}] e^{-sL} \df{L}.
        \end{equation}
        We integrate the right and left sides separately. We first observe that
        \begin{equation}\label{multi-delay-laplace-sub}
            \int_0^\infty \EE[N_{k,L+\ell_n}] e^{-sL} \df{L} = e^{\ell_n s} \int_0^\infty \EE[N_{k,L+\ell_n}] e^{-s(L+\ell_n)} \df{L} = e^{\ell_n s} \varphi(s).
        \end{equation}
        Then, by integration by parts, the left side of \cref{multi-laplace} is equal to       
        \begin{align}
            \int_0^\infty \ddf{L} &\Big[ (L + \rho) \EE[N_{k,L+\ell_n}] \Big] e^{-sL} \df{L} \\
            &= s \int_0^\infty (L+\rho) \EE[N_{k,L+\ell_n}] e^{-sL} \df{L} - \rho \EE[N_{k,\ell_n}] \nonumber\\
            \linechecktext{&= s L \int_0^\infty \EE[N_{k,L+\ell_n}] e^{-sL} \df{L} + s \rho \int_0^\infty  \EE[N_{k,L+\ell_n}] e^{-sL} \df{L} - \rho  \EE[N_{k,\ell_n}] \nonumber\\}
            &= -s \cdot \ddf{s} \left[ \int_0^\infty  \EE[N_{k,L+\ell_n}] e^{-sL} \df{L} \right] + s \rho \int_0^\infty  \EE[N_{k,L+\ell_n}] e^{-sL} \df{L} - \rho  \EE[N_{k,\ell_n}] \nonumber\\
            &= -s \cdot \ddf{s} [e^{\ell_n s} \varphi(s) ] + s \rho e^{\ell_n s} \varphi(s) - \rho  \EE[N_{k,\ell_n}] \qquad\qquad \tag{\cref{multi-delay-laplace-sub}}\nonumber\\
            &= -s e^{\ell_n s} ( \varphi'(s) + \overline{\ell} \varphi(s) ) - \rho  \EE[N_{k,\ell_n}]. \label{multi-lap-LHS}
        \end{align}
To evaluate the right hand side of~\cref{multi-laplace}, first observe that
        \begin{align*}
            \int_0^\infty  \EE[N_{k,L+\rho_i}] e^{-sL} \df{L} &= e^{\rho_i s} \int_0^\infty  \EE[N_{k,L+\rho_i}] e^{-s(L+\rho_i)} \df{L}\\ 
            &= e^{\rho_i s} \left( \int_{\rho_i}^{\ell_n}  \EE[N_{k,L}] e^{-sL} \df{L} + \int_{\ell_n}^\infty  \EE[N_{k,L}] e^{-sL} \df{L}  \right) \\
            &= e^{\rho_i s} \left( P_{i;k} (s) + \varphi(s) \right),
        \end{align*}
        where $P_{i;k} (s)$ is as defined in \cref{multi-Pik-defn}. The right side of \cref{multi-laplace} then simplifies to
        \begin{equation}\label{multi-lap-RHS}
            \frac{q_k}{s} + 2 \sum_{i=1}^n q_i \cdot e^{\rho_i s} (P_{i;k}(s) + \varphi(s)) .
        \end{equation}
        Equating \cref{multi-lap-LHS} and \cref{multi-lap-RHS} yields
        \[ -s e^{\ell_n s} ( \varphi'(s) + \overline{\ell} \varphi(s) ) - \rho  \EE[N_{k,\ell_n}] = \frac{q_k}{s} + 2 \sum_{i=1}^n q_i \cdot e^{\rho_i s} (P_{i;k}(s) + \varphi(s)) . \]
        Rearranging and multiplying the equation by $\frac{e^{-\rho s}}{s}$ yields
        \[ e^{\overline{\ell} s} (\varphi'(s) + \overline{\ell} \varphi(s)) + \frac{2e^{\overline{\ell} s} \varphi(s)}{s} \sum_{i = 1}^n q_i e^{-\ell_i s} + \frac{e^{-\rho s}}{s^2} \left( q_k + s \rho  \EE[N_{k,\ell_n}] \right) + \frac{2e^{\overline{\ell} s}}{s} \sum_{i = 1}^n q_i e^{-\ell_i s} P_{i:k}(s) = 0. \]
        Consider $G_k(s)$ as in \cref{multi-Gk-defn}, and note that $\frac{G_k(s)}{s^2}$ is exactly the constant term in the above first order differential equation for $\varphi(s)$. That is,
        \[ e^{\overline{\ell} s} (\varphi'(s) + \overline{\ell} \varphi(s)) + \frac{2e^{\overline{\ell} s} \varphi(s)}{s} \sum_{i = 1}^n q_i e^{-\ell_i s} + \frac{G_k(s)}{s^2} = 0. \]
        Finally, substituting in $w(s) = e^{\overline{\ell} s} \varphi(s)$ yields the desired result.
    \end{proof}

    We will solve the above differential equation for $w(s)$. To ensure our calculations are well defined, the integral in \cref{multi-constant-eqn} must actually exist. Verifying this is a routine calculus exercise, which we omit for brevity. 
    \begin{obs}\label{multi-existence-of-constant}
        The following integral is finite:
        \[ \alpha_k := \int_0^\infty G_k(t) \exp\left(-2 \sum_{i=1}^n q_i \Ein (\ell_i t)\right) \df{t}. \]
    \end{obs}
    We proceed to analyze the behavior of the Laplace transform of $\EE[N_{k,L}]$ around $s = 0$:
    \begin{lemma}
        As $s \rightarrow 0^+$, we have that
        \[ \int_0^\infty \EE[N_{k,L}] e^{-sL} \df{L} \sim \frac{\alpha_k}{s^2}. \]
    \end{lemma}
    \begin{proof}
    Let $w(s) = e^{\overline{\ell} s} \varphi(s)$, as in    ~\cref{multi-laplace-diffeq}. The maximum number of type $k$ segments we can place on an interval of length $L$ is $\frac{L}{\ell_k}$, so $ \EE[N_{k,L}] \leq \frac{L}{\ell_k}$, and
        \[ w(s) = e^{\overline{\ell} s} \varphi(s) \leq e^{\overline{\ell} s} \int_{\ell_n}^\infty \frac{L}{\ell_k} \cdot e^{-sL} \df{L} = \frac{e^{-\rho s}}{\ell_k} \left(\frac{\ell_n}{s} + \frac{1}{s^2}\right).\] 
        This implies the initial condition $\lim_{s \rightarrow \infty} w(s) = 0.$

        By \cref{multi-laplace-diffeq}, $w(s)$ satisfies the differential equation $$ w'(s) + \frac{2 w(s)}{s} \sum_{i=1}^n q_i e^{-\ell_i s} + \frac{G_k(s)}{s^2} = 0.$$

        \linechecktext{
        We now solve for $w(s)$ using the method of integrating factors. Define the function $\iota(s)$ as
        \[ \iota(s) := \exp \left( 2 \int_\infty^s \frac{1}{u} \sum_{i=1}^n q_i e^{-\ell_i u} \df{u} \right) \qquad \text{so that} \qquad \frac{\iota'(s)}{\iota(s)} = \ddf{s} \ln \iota (s) = \frac{2}{s} \sum_{i=1}^n q_i e^{-\ell_i s}.\]
        Then,
        \[ \ddf{s} ( w(s) \iota(s) ) = \iota(s) \left( w'(s) + \frac{\iota'(s)}{\iota(s)} w(s) \right) = \iota(s) \left( w'(s) + \frac{2 w(s)}{s} \sum_{i=1}^n q_i e^{-\ell_i s} \right).\]
        Multiplying \cref{multi-laplace-diffeq-2} by $\iota(s)$ and simplifying yields
        \[ \ddf{s} ( w(s) \iota(s) ) = - \frac{\iota(s) G_k(s)}{s^2}, \]
        and we may integrate this expression (with a lower bound at positive $\infty$) to solve for $w(s)$:
        \begin{equation}\label{multi-w-sol}
            w(s) \iota(s) = - \int_\infty^s \frac{\iota(t) G_k(t)}{t^2}  \df{t} + C, \qquad w(s) = \int_s^\infty \frac{1}{t^2} \cdot \frac{\iota(t)}{\iota(s)} G_k(t) \df{t} + \frac{C}{\iota(s)}.
        \end{equation}
        Upon inspection, the limit $\lim_{s \rightarrow \infty} \iota(s) = 1$, so to satisfy our initial condition, $C = 0$. Moreover,
        \begin{align*}
            \frac{\iota(t)}{\iota(s)} &= \exp\left(2 \int_s^t \frac{1}{u} \sum_{i=1}^n q_i e^{-\ell_i u} \df{u} \right) = \exp \left( 2 \int_s^t \frac{1}{u} \df{u} -2 \sum_{i=1}^n q_i \int_s^t  \frac{1 - e^{-\ell_i u}}{u} \df{u} \right) \\
            &=\frac{t^2}{s^2} \exp\left(-2 \sum_{i=1}^n q_i \int_s^t  \frac{1 - e^{-\ell_i u}}{u} \df{u}\right).
        \end{align*}
        Substituting this into \cref{multi-w-sol} yields
        \begin{align*}
            w(s) 
            &= \frac{1}{s^2} \int_s^\infty G_k(t) \exp\left(-2 \sum_{i=1}^n q_i \int_s^t \frac{1 - e^{-\ell_i u}}{u} \df{u}\right) \df{t}.
        \end{align*}}
        Using our initial condition and the method of integrating factors, we may solve this first order linear differential equation, which yields
        \[ w(s) = \frac{1}{s^2} \int_s^\infty G_k(t) \exp\left(-2 \sum_{i=1}^n q_i \int_s^t \frac{1 - e^{-\ell_i u}}{u} \df{u}\right) \df{t}. \]

        Thus, as $s \rightarrow 0^+$, by dominated convergence we have
        \[ w(s) \sim \frac{1}{s^2} \int_0^\infty G_k(t) \exp\left(-2 \sum_{i=1}^n q_i \int_0^t \frac{1 - e^{-\ell_i u}}{u} \df{u}\right) \df{t} = \frac{\alpha_k}{s^2}, \]
        with $\alpha_k$ defined as the expression in \cref{multi-existence-of-constant}. Because $\varphi(s) = e^{\ol \ell s} w(s)$, as $s\rightarrow 0^+$ we have $\varphi(s) \sim \frac{\alpha_k}{s^2}$ as well.
        
        Again applying our simple bound $0 \leq \EE[N_{k,L}] \leq \frac{L}{\ell_k},$ we see that as $s \rightarrow 0^+$, $\varphi(s) \leq \int_0^\infty \EE[N_{k,L}] e^{-sL} \df{L} \leq \frac{\ell_n^2}{2 \ell_k} + \varphi(s), $
        and so $\int_0^\infty \EE[N_{k,L}] e^{-sL} \df{L} \sim \frac{\alpha_k}{s^2}$.
    \end{proof}

    Applying~\cref{Tauberian} with $H  = \alpha_k, \beta = 2$, the above implies that 
    $$\int_0^L \EE[N_{k,t}] \df{t} \sim \frac{\alpha_k}{2} \cdot L^2. $$
    Recall that by \cref{multi-recurrence},
    \[ \EE[N_{k,L}] = \frac{1}{L-\overline{\ell}}\left( q_k(L-\ell_k) + 2\sum_{i=1}^n q_i\int_0^{L-\ell_i} \EE[N_{k,t}] \df{t} \right) . \]
    Substituting in our asymptotic for the integral of $\EE[N_{k,L}]$ into the right side of this equation completes the proof that $\EE[N_{k,L}] \sim \alpha_k L.$
\end{proof}

\begin{remark}
    We can derive the variance of $N_{k,L}$ by using a similar method to study the second moment, $\EE[N_{k,L}^2]$. Similar to \cref{multi-recurrence}, one can derive the following recurrence:
    \begin{align*}
        \EE[N_{k,L}^2] = &\frac{1}{L-\overline{\ell}} \Big( q_k (L - \ell_k) + 2 q_k \int_0^{L-\ell_k} \EE[N_{k,t}] \df{t} \\
        &+ \sum_{i = 1}^n 2q_i \int_0^{L-\ell_i} \left( \EE[N_{k,t}] \EE[N_{k,L-\ell_i-t}] + \EE[N_{k,t}^2]\right) \df{t} \Big).
    \end{align*}
    This admits an analysis by Laplace transforms, but the corresponding calculation is even more involved.
\end{remark}

\section{General Length Distributions}\label{ldf} % ldf

In this section, we consider more general $\nu$-RSA processes (c.f.~\cref{RSA-defn}) that do not necessarily have discrete support. In particular, we study the behavior of $\EE [S_L]$ as $L \rightarrow \infty$, where the random variable $S_L$ measures the amount of empty space (not occupied by segments) at saturation. Similar to our study of multidisperse RSA, we begin by deriving a general integral recurrence formula for $\EE[S_L]$.

\begin{remark}\label{ldf-remark}
    In \cite{burridge2004recursive}, Burridge and Mao derive a similar recurrence equation for length distributions with finite support (i.e. there exists $C$ for which $\ell > C$ implies $\nu(\ell) = 0$). They then use this recurrence to study a specific case of bidisperse RSA (c.f. \cref{bidisperse-remark}). Here, we extend their recurrence to general length distributions, and we provide a proof for completeness.
\end{remark}

\begin{proposition}\label{ldf-integral-recurrence}
    Let $\nu(\ell)$ be an ldf (see~\cref{ldf-defn}). Then, in the $\nu$-RSA process, for all $L \geq 0$,
    \[ \left( \int_0^L Z_\nu(t) \df{t} \right) \EE[S_L] = 2 \int_0^{L} \EE[S_t] Z_\nu(L-t) \df{t}. \]
\end{proposition}

\begin{proof}
    We must have $Z_\nu (t) = 0$ for $t < 1$, so when $L \leq 1$, the proposition reduces to $0 = 0$. Thus consider interval length $L > 1$.

    Let $\mu_L(\ell)$ be the probability density function for the length of the segment that will be parked first. In the $\nu$-RSA process, segment lengths are chosen according to a distribution proportional to $\nu(\ell)$. Moreover, the segment is then parked successfully on the interval with probability $\frac{L-\ell}{L}$. Thus, $\mu_L(\ell) \propto \nu(\ell) (L-\ell)$. Normalizing this distribution then yields
    \[ \mu_L(\ell) = \frac{(L-\ell)\nu(\ell)}{\int_0^L (L-t) \nu(t) \df{t}}.\]

    We now compute $\EE[S_L]$. Given the length $\ell$ of the first segment parked, we use an identical argument to the proof of \cref{multi-recurrence} to show that $\EE[S_L] = \frac{2}{L-\ell} \int_0^{L-\ell} \EE[S_t] \df{t}. $

    Using our probability function $\mu_L(\ell)$, we now integrate over the possible values of $\ell$:
    \[ \EE[S_L] = \int_0^L \left( \frac{(L-\ell)\nu(\ell)}{\int_0^L (L-t) \nu(t) \df{t}} \cdot \frac{2}{L-\ell} \int_0^{L-\ell} \EE[S_t] \df{t} \right) \df{\ell}, \]
    which simplifies to
    \begin{equation}\label{ldf-int-rec-raw}
        \left( \int_0^L (L-t) \nu(t) \df{t} \right) \EE[S_L] = 2\int_0^L \int_0^{L-\ell} \nu(\ell) \EE[S_t] \df{t} \df{\ell}.
    \end{equation}
    Switching the order of integration on the right side yields
    \[ 2\int_0^L \int_0^{L-\ell} \nu(\ell) \EE[S_t] \df{t} \df{\ell} = 2\int_0^L \EE[S_t] \int_0^{L-t} \nu(\ell) \df{\ell} \df{t} = 2\int_0^L \EE[S_t] Z_\nu (L-t) \df{t}. \] 
    Meanwhile, by integration by parts, $\int_0^L (L-t) \nu(t) \df{t} \linechecktext{ = L Z_\nu(L) - \int_0^L t \nu(t) \df{t} } = \int_0^L Z_\nu(t) \df{t}.$ Substituting this into \cref{ldf-int-rec-raw} then proves the proposition.
\end{proof}

Using Laplace transforms with the recurrence equation in \cref{ldf-integral-recurrence}, we will later find that $\EE[S_L]$ grows linearly for a general class of convergent ldfs, whereas it grows sublinearly for a general class of divergent ldfs  (see~\cref{d:conv-div}).

\subsection{Convergent Length Distributions}\label{conv}

We first consider the $\nu$-RSA process for convergent ldfs $\nu(\ell)$. When considering such ldfs, we always assume $\int_1^\infty \nu(\ell) \df{\ell} = 1 $ without any loss of generality.

Note as $L \rightarrow \infty$, we must have $Z_{\nu}(L) \rightarrow 1$. If $Z_{\nu}(L) = 1 - o( L^{-\eps} )$ for any $\eps > 0 $, the improper integral $ \int_1^\infty \frac{1-Z_{\nu}(L)}{L} \df{L} $ is finite. Our analysis will rely on the weak condition that this integral is finite, which holds for many natural length distributions (c.f. \cref{conv-ex}).

\begin{theorem}\label{conv-thm}
    If $\nu(\ell)$ is a convergent ldf such that $ \int_1^\infty \frac{1-Z_\nu(L)}{L} \df{L} < \infty,$ then in the $\nu$-RSA process, there exists positive constant $\alpha_\nu$ such that as $L \rightarrow \infty$,
    \[ \EE[ S_L ] \sim \alpha_\nu L. \]
\end{theorem}

\begin{example}\label{conv-ex}
    In the following cases, \cref{conv-thm} applies and $\EE[S_L]$ grows linearly with $L$:
    \begin{enumerate}
        \item The ldf $\nu(\ell)$ represents a finite distribution, i.e. there exists some $C$ for which $\ell > C$ implies $\nu(\ell) = 0$.
        \item The ldf is given by $\nu(\ell) \propto \ell^p$ for $\ell > 1$ with some fixed $p < -1$ .
        \item The ldf is given by $\nu(\ell) \propto e^{-a\ell}$ for any $\ell > 1$ for fixed  $a > 0$.
    \end{enumerate}
\end{example}

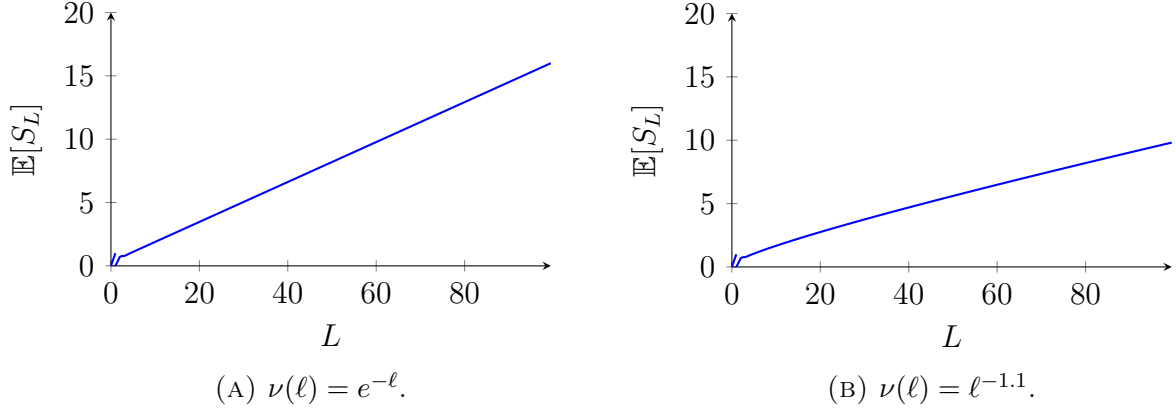
\begin{figure}
    \centering
    \begin{subfigure}{.5\textwidth}
        \begin{tikzpicture}[scale=1]
            \begin{axis} [axis lines = left,
                xlabel = \(L\),
                ylabel = {$\EE[S_L]$},
                width=.9\textwidth,
                height=.6\textwidth,
                ymin=0,
                ymax=20
            ]
            \addplot[color=blue,thick] coordinates {
                (0,0)(1,1)
            };
            \addplot [color=blue, thick] coordinates { 
                (1,0)(1.5,0.333892)(2,0.705897)(2.5,0.777425)(3,0.7786)(3.5,0.851016)(4,0.942233)(4.5,1.02581)(5,1.10486)(5.5,1.18333)(6,1.26203)(6.5,1.34084)(7,1.41967)(7.5,1.49848)(8,1.57729)(8.5,1.65608)(9,1.73487)(9.5,1.81366)(10,1.89245)(10.5,1.97124)(11,2.05003)(11.5,2.12881)(12,2.2076)(12.5,2.28639)(13,2.36517)(13.5,2.44396)(14,2.52275)(14.5,2.60153)(15,2.68032)(15.5,2.75911)(16,2.83789)(16.5,2.91668)(17,2.99547)(17.5,3.07425)(18,3.15304)(18.5,3.23183)(19,3.31061)(19.5,3.3894)(20,3.46818)(20.5,3.54697)(21,3.62576)(21.5,3.70454)(22,3.78333)(22.5,3.86212)(23,3.9409)(23.5,4.01969)(24,4.09848)(24.5,4.17726)(25,4.25605)(25.5,4.33484)(26,4.41362)(26.5,4.49241)(27,4.5712)(27.5,4.64998)(28,4.72877)(28.5,4.80755)(29,4.88634)(29.5,4.96513)(30,5.04391)(30.5,5.1227)(31,5.20149)(31.5,5.28027)(32,5.35906)(32.5,5.43785)(33,5.51663)(33.5,5.59542)(34,5.67421)(34.5,5.75299)(35,5.83178)(35.5,5.91056)(36,5.98935)(36.5,6.06814)(37,6.14692)(37.5,6.22571)(38,6.3045)(38.5,6.38328)(39,6.46207)(39.5,6.54086)(40,6.61964)(40.5,6.69843)(41,6.77722)(41.5,6.856)(42,6.93479)(42.5,7.01358)(43,7.09236)(43.5,7.17115)(44,7.24993)(44.5,7.32872)(45,7.40751)(45.5,7.48629)(46,7.56508)(46.5,7.64387)(47,7.72265)(47.5,7.80144)(48,7.88023)(48.5,7.95901)(49,8.0378)(49.5,8.11659)(50,8.19537)(50.5,8.27416)(51,8.35294)(51.5,8.43173)(52,8.51052)(52.5,8.5893)(53,8.66809)(53.5,8.74688)(54,8.82566)(54.5,8.90445)(55,8.98324)(55.5,9.06202)(56,9.14081)(56.5,9.2196)(57,9.29838)(57.5,9.37717)(58,9.45596)(58.5,9.53474)(59,9.61353)(59.5,9.69231)(60,9.7711)(60.5,9.84989)(61,9.92867)(61.5,10.0075)(62,10.0862)(62.5,10.165)(63,10.2438)(63.5,10.3226)(64,10.4014)(64.5,10.4802)(65,10.559)(65.5,10.6378)(66,10.7165)(66.5,10.7953)(67,10.8741)(67.5,10.9529)(68,11.0317)(68.5,11.1105)(69,11.1893)(69.5,11.268)(70,11.3468)(70.5,11.4256)(71,11.5044)(71.5,11.5832)(72,11.662)(72.5,11.7408)(73,11.8195)(73.5,11.8983)(74,11.9771)(74.5,12.0559)(75,12.1347)(75.5,12.2135)(76,12.2923)(76.5,12.3711)(77,12.4498)(77.5,12.5286)(78,12.6074)(78.5,12.6862)(79,12.765)(79.5,12.8438)(80,12.9226)(80.5,13.0013)(81,13.0801)(81.5,13.1589)(82,13.2377)(82.5,13.3165)(83,13.3953)(83.5,13.4741)(84,13.5529)(84.5,13.6316)(85,13.7104)(85.5,13.7892)(86,13.868)(86.5,13.9468)(87,14.0256)(87.5,14.1044)(88,14.1831)(88.5,14.2619)(89,14.3407)(89.5,14.4195)(90,14.4983)(90.5,14.5771)(91,14.6559)(91.5,14.7346)(92,14.8134)(92.5,14.8922)(93,14.971)(93.5,15.0498)(94,15.1286)(94.5,15.2074)(95,15.2862)(95.5,15.3649)(96,15.4437)(96.5,15.5225)(97,15.6013)(97.5,15.6801)(98,15.7589)(98.5,15.8377)(99,15.9164)(99.5,15.9952)
            };
            \end{axis}
        \end{tikzpicture}
        \caption{$\nu(\ell) = e^{-\ell}$.}
    \end{subfigure}%
    \begin{subfigure}{.5\textwidth}
        \begin{tikzpicture}[scale=1]
            \begin{axis} [axis lines = left,
                xlabel = \(L\),
                ylabel = {$\EE[S_L]$},
                width=.9\textwidth,
                height=.6\textwidth,
                ymin=0,
                xmin=0,
                ymax=20
            ]
            \addplot[color=blue,thick] coordinates {
                (0,0)(1,1)
            };
            \addplot [color=blue, thick] coordinates { 
                (1,0)(1.5,0.332892)(2,0.696728)(2.5,0.767484)(3,0.771534)(3.5,0.835439)(4,0.913682)(4.5,0.985535)(5,1.053)(5.5,1.11865)(6,1.18315)(6.5,1.24664)(7,1.30919)(7.5,1.37088)(8,1.43178)(8.5,1.49195)(9,1.55146)(9.5,1.61036)(10,1.66868)(10.5,1.72647)(11,1.78375)(11.5,1.84057)(12,1.89694)(12.5,1.95289)(13,2.00845)(13.5,2.06363)(14,2.11846)(14.5,2.17294)(15,2.2271)(15.5,2.28095)(16,2.3345)(16.5,2.38777)(17,2.44076)(17.5,2.4935)(18,2.54598)(18.5,2.59821)(19,2.65022)(19.5,2.702)(20,2.75355)(20.5,2.8049)(21,2.85605)(21.5,2.90699)(22,2.95775)(22.5,3.00832)(23,3.05871)(23.5,3.10892)(24,3.15896)(24.5,3.20884)(25,3.25856)(25.5,3.30812)(26,3.35753)(26.5,3.40679)(27,3.4559)(27.5,3.50488)(28,3.55371)(28.5,3.60241)(29,3.65098)(29.5,3.69942)(30,3.74774)(30.5,3.79593)(31,3.844)(31.5,3.89196)(32,3.93979)(32.5,3.98752)(33,4.03514)(33.5,4.08265)(34,4.13005)(34.5,4.17734)(35,4.22454)(35.5,4.27164)(36,4.31863)(36.5,4.36553)(37,4.41234)(37.5,4.45905)(38,4.50567)(38.5,4.55221)(39,4.59865)(39.5,4.64501)(40,4.69128)(40.5,4.73747)(41,4.78357)(41.5,4.8296)(42,4.87554)(42.5,4.9214)(43,4.96719)(43.5,5.0129)(44,5.05854)(44.5,5.1041)(45,5.14959)(45.5,5.19501)(46,5.24036)(46.5,5.28564)(47,5.33084)(47.5,5.37599)(48,5.42106)(48.5,5.46607)(49,5.51101)(49.5,5.55589)(50,5.6007)(50.5,5.64546)(51,5.69015)(51.5,5.73478)(52,5.77935)(52.5,5.82386)(53,5.86831)(53.5,5.91271)(54,5.95704)(54.5,6.00132)(55,6.04555)(55.5,6.08972)(56,6.13383)(56.5,6.1779)(57,6.2219)(57.5,6.26586)(58,6.30976)(58.5,6.35362)(59,6.39742)(59.5,6.44117)(60,6.48487)(60.5,6.52852)(61,6.57213)(61.5,6.61568)(62,6.65919)(62.5,6.70265)(63,6.74606)(63.5,6.78943)(64,6.83276)(64.5,6.87603)(65,6.91926)(65.5,6.96245)(66,7.0056)(66.5,7.0487)(67,7.09176)(67.5,7.13477)(68,7.17775)(68.5,7.22068)(69,7.26357)(69.5,7.30642)(70,7.34922)(70.5,7.39199)(71,7.43472)(71.5,7.47741)(72,7.52006)(72.5,7.56267)(73,7.60525)(73.5,7.64778)(74,7.69028)(74.5,7.73274)(75,7.77516)(75.5,7.81755)(76,7.8599)(76.5,7.90221)(77,7.94449)(77.5,7.98673)(78,8.02894)(78.5,8.07111)(79,8.11325)(79.5,8.15535)(80,8.19742)(80.5,8.23946)(81,8.28146)(81.5,8.32343)(82,8.36537)(82.5,8.40727)(83,8.44914)(83.5,8.49098)(84,8.53279)(84.5,8.57456)(85,8.6163)(85.5,8.65802)(86,8.6997)(86.5,8.74135)(87,8.78297)(87.5,8.82456)(88,8.86612)(88.5,8.90765)(89,8.94915)(89.5,8.99062)(90,9.03207)(90.5,9.07348)(91,9.11486)(91.5,9.15622)(92,9.19755)(92.5,9.23885)(93,9.28012)(93.5,9.32136)(94,9.36258)(94.5,9.40377)(95,9.44493)(95.5,9.48606)(96,9.52717)(96.5,9.56825)(97,9.60931)(97.5,9.65034)(98,9.69134)(98.5,9.73232)(99,9.77327)(99.5,9.81419)
            };
            \end{axis}
        \end{tikzpicture}
        \caption{$\nu(\ell) = \ell^{-1.1}$.}
    \end{subfigure}
    \caption{$\EE[S_L]$ grows linearly under various convergent ldfs.}
\end{figure}

\begin{proof}[Proof of \cref{conv-thm}:]
    Let $\varphi := \Lap \{ \EE[S_L] \}$ be the Laplace transform of $\EE[S_L]$ so that
    \[ \varphi(s) = \int_0^\infty \EE[S_L] e^{-sL} \df{L}. \]
    Define $V : \RR^{>0} \rightarrow \RR$ as $V := \Lap \{ 1-Z_\nu(L) \}$, and define the function $V_\star : \RR^{>0} \rightarrow \RR$ as a modified version of $V$,
    \begin{equation}\label{conv-V-defn}
        V(s) := \int_0^\infty (1-Z_\nu(L)) e^{-s L} \df{L}, \qquad V_\star(s) := \int_1^\infty (1-Z_\nu(L)) e^{-sL} \df{L}.
    \end{equation}
    Since $Z_\nu (L) = 0$ for $L < 1$, we have $V(s) = \frac{1-e^{-s}}{s} + V_\star(s).$ We will first use these functions to derive differential inequalities satisfied by $\varphi(s)$.

    \begin{lemma}\label{conv-diffineq-lower}
        The following differential inequalities hold when $s > 0$:
        \[ -\varphi'(s) > \frac{2}{s} ( e^{-s} - s V_\star(s) ) \varphi(s), \qquad -\varphi'(s) < - \left( \frac{ V(s) }{ s } \right)' + \frac{2}{s} (e^{-s} - s V_\star(s)) \varphi(s). \]
    \end{lemma}

    \begin{proof}
        By \cref{ldf-integral-recurrence},
        \begin{equation}\label{ldf-diffineq-recurrence}
            \left( \int_0^L Z_\nu(t) \df{t} \right) \EE[S_L] = 2 \int_0^{L} \EE[S_t] Z_\nu(L-t) \df{t}.
        \end{equation}
        Because $Z_\nu(t) \leq 1$, we have $\int_0^L Z_\nu (t) \df{t} < L$, and
        \begin{equation}\label{ldf-v-eq}
            L \EE[S_L] > 2 \int_0^{L} \EE[S_t] Z_\nu(L-t) \df{t} = 2 \int_0^L \EE[S_t] \df{t} - 2 \int_0^{L} \EE[S_t] (1 - Z_\nu(L-t)) \df{t},
        \end{equation}

        Taking Laplace transforms, we have that by the time-multiplication property (c.f. \cref{Laplace-Properties}), $\Lap\{ L \EE[S_L] \} = - \varphi'(s).$ Moreover, the integration property implies \\ $\Lap \left\{ 2 \int_0^L \EE[S_t] \df{t} \right\} = \frac{2\varphi(s)}{s}$, and the convolution property implies
        \[ \Lap \left\{ - 2 \int_0^{L} \EE[S_t] (1 - Z_\nu (L-t)) \df{t}\right\} = -2 \varphi(s) V(s).\] Taking the Laplace transform of \cref{ldf-v-eq} now yields $ -\varphi'(s) > \frac{2 \varphi(s)}{s} ( 1 - sV(s) ), $ and substituting $V_\star(s)$ for $V(s)$ yields the first differential inequality.

        To derive the second inequality, we reuse \cref{ldf-diffineq-recurrence}. This time, we note $\int_0^L Z_\nu (t) \df{t} = L - \int_0^L (1 - Z_\nu(t)) \df{t}$, so
        \[ L \EE[S_L] = \left( \int_0^L (1 - Z_\nu(t)) \df{t} \right) \EE[S_L] + 2 \int_0^{L} \EE[S_t] Z_\nu(L-t) \df{t}. \]
        By definition, $S_L \leq L$, so $L \EE[S_L] \leq L \int_0^L (1 - Z_\nu(t)) \df{t} + 2 \int_0^L \EE[S_t] Z_\nu(L-t) \df{t},$ or equivalently,
        \begin{equation}\label{ldf-v-eq-2}
            L \EE[S_L] \leq L \int_0^L (1 - Z_\nu(t)) \df{t} + 2 \int_0^L \EE[S_t] \df{t} - 2 \int_0^L \EE[S_t] (1 - Z_\nu(L-t)) \df{t}.
        \end{equation}
        Taking Laplace transforms, we note that
        \[ \Lap\left\{ L \int_0^L (1 - Z_\nu(t)) \df{t} \right\} = -\ddf{s} \Lap\left\{ \int_0^L (1 - Z_\nu(t)) \df{t} \right\} = - \ddf{s} \left( \frac{ V(s) }{ s } \right). \]
        Taking the Laplace transform of \cref{ldf-v-eq-2} then yields
        \[ -\varphi'(s) \leq - \left( \frac{ V(s) }{ s } \right)' +  \frac{2\varphi(s)}{s} - 2 \varphi(s) V(s), \]
        and substituting in $V_\star (s)$ as before gives the second differential inequality.
    \end{proof}

    We now find functions that use the differential inequalities to directly bound $\varphi(s)$. If we treat the first inequality in \cref{conv-diffineq-lower} as an equality, then the solution to the differential equation is the function $b : \RR^{> 0} \rightarrow \RR$, given by 
    \begin{equation}\label{conv-br-defn}
        b(s) := \frac{r(s)}{s^2}, \qquad \text{where } r(s) := \exp\left( 2 \int_1^s \frac{1-e^{-t}}{t} \df{t} + 2 \int_1^s V_\star(t) \df{t} \right).
    \end{equation}
    We will show $b(s)$ yields a lower bound on $\varphi(s)$.

    \begin{lemma}\label{conv-lap-lower}
        Fix $\eta > 0$. Then, for $s \in (0, \eta)$,
        \[ \varphi(s) \geq \frac{\varphi(\eta)}{b(\eta)} \cdot b(s).\]
    \end{lemma}
    \begin{proof}
        Let $\varphi_\star(s) = \frac{\varphi(\eta)}{b(\eta)} \cdot b(s)$. We may check that
        $\varphi_\star (\eta) = \varphi(\eta)$ and that $\varphi_\star$ is a solution to the differential equation
        \[ -\varphi_\star'(s) = \frac{2}{s} ( e^{-s} - s V_\star(s) ) \varphi_\star(s). \]

        By \cref{conv-diffineq-lower}, $\varphi(s)$ satisfies $-\varphi'(s) > \frac{2}{s} ( e^{-s} - s V_\star(s) ) \varphi(s),$
        and we quickly note $ V_\star (s) \leq \int_1^\infty 1 \cdot e^{-s L} \df{L} = \frac{e^{-s}}{s},$ which implies $e^{-s} - s V_\star(s) > 0$. Therefore, \cref{negdiffineq} holds here and implies that when $s \in (0,\eta)$, $\varphi(s) \geq \varphi_\star(s)$.
    \end{proof}

    Define function $h_\eta(s)$ as
    \begin{equation}\label{conv-heta-defn}
        h_\eta(s) := \int_s^\eta -\left( \frac{V(t)}{t} \right)' \cdot \frac{b(s)}{b(t)} \df{t}.
    \end{equation}
    The function $h_\eta(s)$ is the error when approximating $\varphi(s)$ with $b(s)$. By adding $h_\eta(s)$ to $b(s)$, we will derive an upper bound for $\varphi(s)$.
    
    \begin{lemma}\label{conv-lap-upper}
        Fix $\eta > 0$. Then, for $s \in (0,\eta)$,
        \[ \varphi(s) \leq \frac{\varphi(\eta)}{b(\eta)} \cdot b(s) + h_\eta(s). \]
    \end{lemma}

    \begin{proof}
        Let $\varphi^\star (s) = \frac{\varphi(\eta)}{b(\eta)} \cdot b(s) + h_\eta(s) $. Because $h_\eta ( \eta ) = 0$, we have
        $\varphi^\star(\eta) = \varphi(\eta)$. Moreover, $\varphi^\star(s)$ satisfies the differential equation
        \[ -\varphi^{\star\prime}(s) = -\left( \frac{ V(s) }{ s } \right)' + \frac{2}{s} (e^{-s} - s V_\star(s)) \varphi^\star(s). \]

        \cref{conv-diffineq-lower} implies $ -\varphi'(s) < - \left( \frac{ V(s) }{ s } \right)' + \frac{2}{s} (e^{-s} - s V_\star(s)) \varphi(s).$ Thus, \cref{negdiffineq} again proves that $\varphi(s) < \varphi^\star(s)$ when $s \in (0, \eta)$, completing the lemma.
    \end{proof}

    We now prove useful results about the functions $r(s)$ and $h_\eta(s)$, which will allow us to bound $\varphi(s)$.

    \begin{lemma}
        The function $r(s)$ is positive, continuous, and increasing when $s \in (0,\infty)$. It also satisfies
        \[ \lim_{s \rightarrow 0^+} r(s) = r(0) > 0. \]
    \end{lemma}

    \begin{proof}
        Recall
        \[r(s) := \exp\left( 2 \int_1^s \frac{1-e^{-t}}{t} \df{t} + 2 \int_1^s V_\star(t) \df{t} \right). \] 
        Note $r(s)$ is continuous (increasing) because the integrals are continuous (increasing), and it is positive by inspection. Moreover, $\frac{1-e^{-t}}{t}$ is bounded on $(0,1]$, so $ \int_{0}^1 \frac{1 - e^{-t}}{t} \df{t}$ exists.
        
        Meanwhile, we have assumed that $\int_1^\infty \frac{1 - Z_\nu(L)}{L} \df{L} < \infty,$ and
        \begin{equation}\label{conv-V-star-int-finite}
            \int_1^\infty \frac{1 - Z_\nu(L)}{L} \df{L} = \int_1^\infty \int_0^\infty  (1 - Z_\nu(L)) e^{-tL} \df{t} \df{L} = \int_0^\infty V_\star (t) \df{t},
        \end{equation}
        where the last integral interchange is allowable due to Fubini's Theorem. Since the limit $ \lim_{s \rightarrow 0^+} \int_{s}^{1} V_\star(t) \df{t}$ is strictly less than the above value, it is finite as well, which proves that $ \lim_{s \rightarrow 0^+ } r(s)$ exists and is equal to $r(0)$, which by definition of $r$ must be positive.
    \end{proof}

    This lemma will allow us to bound the approximation error $h_\eta(s)$ by a function that converges to $0$ as $\eta$ becomes small.
    \begin{lemma}\label{conv-p-up}
        There exist continuous functions $H_\eta : [0, \eta] \rightarrow \RR$ for each $\eta > 0$ that satisfy $\lim_{\eta \rightarrow 0^+} H_\eta (0) = 0$ and $ s^2 h_\eta(s) \leq H_\eta(s)$ when $s \in (0,\eta)$.
    \end{lemma}
    \begin{proof}
        Fix $\eta > 0$. First, note $-\left( \frac{V(s)}{s} \right)'$ is positive, as $V(s)$ is a decreasing function (c.f. \cref{conv-V-defn}). Then, for all $0 < s < \eta$,
        \begin{align}
            h_\eta(s) &= \frac{r(s)}{s^2} \int_s^\eta -\left( \frac{V(t)}{t} \right)' \cdot \frac{t^2}{r(t)} \df{t} \tag{\cref{conv-br-defn}, \cref{conv-heta-defn}} \nonumber\\ 
            &\leq \frac{r(s)}{s^2} \cdot \int_s^\eta - \left( \frac{V(t)}{t} \right)' \cdot \frac{t^2}{r(0)} \df{t}, \tag{$r(s)$ is increasing}\nonumber\\
            h_\eta(s) &\leq \frac{1}{s^2} \cdot \frac{r(s)}{ r(0) } \cdot \int_s^\eta - \left( \frac{V(t)}{t} \right)' \cdot t^2 \df{t}.\label{conv-hetabound}
        \end{align}
        We may use integration by parts on the integral, which yields
        \[ \int_s^\eta - \left( \frac{V(t)}{t} \right)' \cdot t^2 \df{t} = sV(s) - \eta V(\eta) + 2 \int_{s}^\eta V(t) \df{t} \leq s V(s) + 2 \int_s^\eta V(t) \df{t}. \]
        Recall $V(t) = \frac{1 - e^{-t}}{t} + V_\star (t) \leq 1 + V_\star (t)$. Substitution yields
        \[ s V(s) + 2 \int_s^\eta V(t) \df{t} \leq s V(s) + 2 (\eta-s) + 2 \int_s^\eta V_\star (t) \df{t}. \]
        Substituting this upper bound back into \cref{conv-hetabound} yields the estimation $ h_\eta(s) \leq \frac{H_\eta(s)}{s^2}, $ with $H_\eta : (0, \eta] \rightarrow \RR$ given by
        \[ H_\eta(s) := \frac{r(s)}{r(0)} \cdot \left( s V(s) + 2 (\eta-s) + 2 \int_s^\eta V_\star (t) \df{t} \right). \]
        The continuity of $H_\eta(s)$ follows by the continuity of $r(s)$ and $V(s)$. Note $H_\eta(s)$ is not yet defined at $0$ since $V$ is not defined at $0$. However, because as $L \rightarrow \infty$, we have $1 - Z_\nu(L) \rightarrow 0$. The abelian final value theorem (c.f. \cref{Laplace-Properties}) then implies $\lim_{s\rightarrow 0^+} s V(s) = 0.$ Moreover, we have previously shown $\int_0^\eta V_\star (t) \df{t}$ is finite (c.f. \cref{conv-V-star-int-finite}). It follows that
        \[ \lim_{s \rightarrow 0^+} H_\eta(s) =  2 \eta + 2 \int_0^\eta V_\star (t) \df{t}. \] 
        Define $H_\eta(0)$ to be this value. By inspection of this formula, $\lim_{\eta \rightarrow 0^+} H_\eta(0) = 0$.
    \end{proof}

    Finally, we analyze the behavior of $\varphi(s)$ as $s \rightarrow 0$.

    \begin{lemma}
        As $s\rightarrow 0^+$, there exists a positive constant $\alpha_\nu$ such that
        \[ \varphi(s) \sim \frac{\alpha_\nu}{s^2}. \]
    \end{lemma}

    \begin{proof}
        Fix any $\eps > 0$. We will now find $l_\eps, \delta_\eps$ that satisfy the conditions of \cref{limitfromupperlower} for our choice of $\eps$, i.e. that for all $0 < s < \delta_\eps$, we have $l_\eps < s^2 \varphi(s) < l_\eps + \eps$.

        We use the functions $H_\eta$ as defined in \cref{conv-p-up}. Choose sufficiently small $\eta$ such that $ H_\eta(0) < \frac{\eps}{3}.$ Then, choose sufficiently small $\delta_1$ such that for all $0 < s < \delta_1$, we have 
        \begin{equation}\label{conv-H-eta}
            H_\eta(s) < H_\eta(0) + \frac{\eps}{3} < \frac{2\eps}{3}.
        \end{equation}
        Finally, because $r(s)$ is increasing, we may choose $\delta_2$ such that for all $0 < s < \delta_2$, 
        \begin{equation}\label{conv-r-bound}
            r(0) < r(s) < r(0) + \frac{b(\eta)}{\varphi(\eta)} \cdot \frac{\eps}{3} .
        \end{equation}

        Let $\delta = \inf\{ \eta, \delta_1, \delta_2 \}$. By \cref{conv-lap-lower}, for $0 < s < \delta$, we have $\varphi(s) \geq \frac{\varphi(\eta)}{b(\eta)} \cdot b(s)$, so
        \[
            s^2 \varphi(s) \geq \frac{\varphi(\eta)}{b(\eta)} \cdot r(s) > \frac{\varphi(\eta)}{b(\eta)} \cdot r(0).
        \]
        Meanwhile, by \cref{conv-lap-upper}, for $0 < s < \delta$, we have $\varphi(s) \leq \frac{\varphi(\eta)}{b(\eta)} \cdot b(s) + h_\eta (s), $ so
        \[ s^2 \varphi(s) \leq \frac{\varphi(\eta)}{b(\eta)} \cdot r(s) + s^2 h_\eta(s) < \frac{\varphi(\eta)}{b(\eta)} \cdot r(0) + \eps, \]
        where the second inequality follows from \cref{conv-r-bound} and from $s^2 h_\eta(s) \leq H_\eta(s) < \frac{2\eps}{3}$ (c.f. \cref{conv-p-up} and \cref{conv-H-eta}). 
        
        Define $\delta_\eps = \delta$ and $l_\eps = \frac{\varphi(\eta)}{b(\eta)} \cdot r(0).$ We have shown that for $0 < s < \delta_\eps$, we have $l_\eps < s^2 \varphi(s) < l_\eps + \eps,$ which satisfies the conditions of \cref{limitfromupperlower} and proves that $\lim_{s \rightarrow 0^+} s^2\varphi(s)$ exists. Denote this limit $\alpha_\nu$. Notably, all the lower bounds $l_\eps$ are positive, so $\alpha_\nu$ must also be positive, concluding the lemma.
    \end{proof}

    By the Hardy-Littlewood Tauberian Theorem (c.f. \cref{Tauberian}), $\varphi(s) \sim \frac{\alpha_\nu}{s^2}$ implies that as $L \rightarrow \infty$,
    \[ \int_0^L \EE[S_t] \df{t} \sim \frac{\alpha_\nu}{2} L^2. \]
    From here, an analytical argument will complete the proof. 

    Pick arbitrary $\eps > 0$. We will first show there exists $L_\ell$ such that $L > L_\ell$ implies $\frac{\EE[S_L]}{L} > \alpha_\nu(1-\eps)$, and we will later show there exists $L_u$ such that $L > L_u$ implies $\frac{\EE[S_L]}{L} < \alpha_\nu (1+ \eps)$, proving $\lim_{L \rightarrow \infty } \frac{\EE[S_L]}{L} = \alpha_\nu$. 
    
    Begin by choosing $L_1$ for which $L > L_1$ implies $\int_0^L \EE[S_t] \df{t} > (1-\eps)^{1/4} \cdot \frac{\alpha_\nu}{2} L^2$.  Moreover, choose $L_2$ for which $L > L_2$ implies $Z_\nu(L) \geq (1 - \eps)^{1/4}$. By \cref{ldf-integral-recurrence}, \[ \left( \int_0^L Z_\nu(t) \df{t} \right) \EE[S_L] = 2 \int_0^{L} \EE[S_t] Z_\nu(L-t) \df{t}. \]
    We always have $Z_\nu(L) < 1$. Therefore,
    \[ L \EE[S_L] \geq \left( \int_0^L Z_\nu(t) \df{t} \right) \EE[S_L] = 2 \int_0^L \EE[S_t] Z_\nu(L-t) \df{t} \geq  2(1-\eps)^{1/4} \int_{0}^{L-L_2} \EE[S_t] \df{t},\]
    where the last equality follows by the definition of $L_2$. Then, when $L - L_2 > L_1,$ we have
    \[2(1-\eps)^{1/4} \int_{0}^{L-L_2} \EE[S_t] \df{t} > (1-\eps)^{1/2} \alpha_\nu (L-L_2)^2.\]
    Choose $L_\ell$ sufficiently large so that $L > L_\ell$ implies $(L-L_2) > (1-\eps)^{1/4} L$, and so that $L_\ell > L_1 + L_2$. Then, for $L > L_\ell$,
    \[ (1-\eps)^{1/2} \alpha_\nu (L-L_2)^2 > (1-\eps) \alpha_\nu L^2.\]
    Thus, for $L > L_\ell$, we have $L\EE[S_L] > (1-\eps) \alpha_\nu L^2$, or $\EE[S_L] > (1-\eps) \alpha_\nu L.$

    We now show there exists $L_u$ such that $L > L_u$ implies $\frac{\EE[S_L]}{L} < \alpha_\nu (1+\eps)$. Choose $L_3$ for which $L > L_3$ implies $\int_0^L \EE[S_t] \df{t} < (1 + \eps)^{1/3} \cdot \frac{\alpha_\nu}{2} L^2$, and choose $L_4$ for which $L > L_4$ implies $Z_\nu(L) > (1 + \eps)^{-1/3} .$ For $L > L_3 + L_4$, we have similarly to before that
    \begin{align*}
        \alpha_\nu(1 + \eps)^{1/3} L^2 &> 2 \int_0^L \EE[S_t] \df{t} \geq 2 \int_0^L \EE[S_t] Z_\nu(L-t) \df{t} = \left( \int_0^L Z_\nu(t) \df{t} \right) \EE[S_L] \\
        &\geq \left( \int_{L_4}^L Z_\nu(t) \df{t} \right) \EE[S_L] > \frac{L-L_4}{(1 + \eps)^{1/3}} \EE[S_L].
    \end{align*}
    Choose $L_u$ sufficiently large so that $L > L_u$ implies $(L-L_4) > \frac{L}{(1 + \eps)^{1/3}}$, and so that $L_u > L_3 + L_4$. Then, $\frac{L-L_4}{(1 + \eps)^{1/3}} \EE[S_L] > \frac{L\EE[S_L]}{(1 + \eps)^{2/3}} .$ Thus, for $L > L_u$, we have $ \alpha_\nu(1 + \eps)^{1/3} L^2 > \frac{L\EE[S_L]}{(1 + \eps)^{2/3}}$, or $ \EE[S_L] < (1 + \eps) \alpha_\nu L.$ This completes the proof of the theorem.
\end{proof}

\subsection{Divergent Length Distributions}\label{div}

In this section we consider the $\nu$-RSA process for a divergent ldf $\nu(\ell)$ (c.f. \cref{ldf-defn}). For any divergent ldf $\nu$, consider its normalizing constant $Z_\nu$. We must have
\[ \frac{\int_0^L t Z_\nu(t) \df{t}}{\int_0^L Z_\nu(t) \df{t}} \geq \frac{L}{2}, \]
because if we treat $Z_\nu$ as an non-normalized probability distribution on $[0,L]$, and we sample random variable $X$ from the distribution, the left hand side is simply $\EE[X]$. But $Z_\nu$ is a strictly increasing function, and so $\EE[X]$ is skewed to above $\frac{L}{2}$. The difference between the left hand side and the right hand side is intuitively a measure of how quickly $Z_\nu$ grows, as when $Z_\nu$ grows very quickly, $\EE[X]$ is skewed higher. Our next theorem shows that $\EE[S_L] = o(L)$ for any divergent ldf that satisfies a slightly stronger version of the above inequality.

\begin{theorem}\label{div-thm}
    Let $\nu$ be a divergent ldf. If there exists $\eps > 0$ such that for sufficiently large $L$, we have $ \frac{\int_0^L t Z_\nu(t) \df{t}}{\int_0^L Z_\nu(t) \df{t}} \geq (1+\eps) \cdot \frac{L}{2},$
    then in the $\nu$-RSA process,
    \[ \EE[S_L] = o(L). \]
\end{theorem}

\begin{example}
 \cref{div-thm} applies to the following natural families of ldfs, where we show $\EE[S_L]$ grows sublinearly with $L$:
    \begin{enumerate}
        \item The ldf $\nu(\ell) = \ell^p$ for all $\ell > 1$ for any fixed $p > -1$ (we give more precise asymptotics for this case in \cref{plaw}).
        \item The the ldf $\nu(\ell) = e^{a\ell}$ for all $\ell > 1$ for any fixed $a > 0$.
    \end{enumerate}
\end{example}

We will prove~\cref{div-thm} using the following lemma:

\begin{lemma}\label{div-lemma}
    Given any integer $n \geq 0$, there exists $b$ such that $L > b$ implies
    \[ \EE[S_L] \leq \left(1 - \frac{\eps}{2}\right)^n L. \] 
\end{lemma}
\begin{proof} We prove the lemma by induction. The lemma immediately holds for $n = 0$: for any $L$, we must have $\EE[S_L] \leq L$, as the empty space $S_L$ cannot exceed $L$.
    
    Given that \cref{div-lemma} holds for $n=k-1$, we will now show it holds for $n=k$. That is, given that there exists $b$ such that $L>b$ implies $\EE[S_L] \leq \left(1 - \frac{\eps}{2}\right)^{k-1} L$, we will show $\EE[S_L] \leq \left(1 - \frac{\eps}{2}\right)^{k} L$ also holds for sufficiently large $L$
    
    For brevity, define $r := \left(1 - \frac{\eps}{2}\right)^{k-1}$. When $L > b$, we have
    \begin{align}
        \left( \int_0^L Z_\nu(t) \df{t} \right) \EE[S_L] &= 2 \int_0^{b} \EE[S_t] Z_\nu(L-t) \df{t} + 2 \int_b^{L} \EE[S_t] Z_\nu(L-t) \df{t}\tag{By \cref{ldf-integral-recurrence}} \nonumber\\
        &\leq 2 \int_0^{b} t Z_\nu(L-t) \df{t} + 2 r \int_b^{L} t Z_\nu(L-t) \df{t} \nonumber\\
        &= 2 r \int_0^L t  Z_\nu(L-t) \df{t} + 2 \left(1- r \right) \int_0^b t Z_\nu(L-t) \df{t}.\label{div-ineq}
    \end{align}
    Note that $\int_0^L t Z_\nu(L-t) \df{t} = L \int_{0}^L Z_\nu(t) \df{t} - \int_{0}^L t Z_\nu(t) \df{t}.$ 
    Let $ L^\star $ be sufficiently large so that $ L > L^\star$ implies $ \frac{\int_0^L t Z_\nu(t) \df{t}}{\int_0^L Z_\nu(t) \df{t}} \geq (1+\eps) \cdot \frac{L}{2},$ as described by the theorem statement. Then, when $L > L^\star$, we may rearrange \cref{div-ineq} into
    \begin{align*}
        \EE[S_L] &\leq 2 r L - (1+ \eps) r L + 2 \left(1-r\right) \cdot \frac{\int_0^b t  Z_\nu (L-t) \df{t}} {\int_0^L Z_\nu(t) \df{t} } \\
        &\leq (1 - \eps) r L  + 2 \left(1-r \right) b.
    \end{align*}
    But $2 (1-r) b$ is constant, so it follows that $\EE[S_L] < r \left(1- \frac{\eps}{2}\right) L = \left(1- \frac{\eps}{2}\right)^{k} L$ for sufficiently large $L$, completing the induction.
\end{proof}

\begin{proof}[Proof of \cref{div-thm}:] For any $\eta > 0$, we may find $n$ such that $ \left(1 - \frac{\eta}{2}\right)^n < \eta$. \cref{div-lemma} then implies that $\EE[S_L] < \eta L$ for sufficiently large $L$. Thus, $\lim_{L \rightarrow \infty} \frac{\EE[S_L]}{L} = 0$, completing the proof of~\cref{div-thm}.
\end{proof}

\begin{figure}
    \centering
    \begin{subfigure}{.5\textwidth}
        \begin{tikzpicture}[scale=1]
            \begin{axis} [axis lines = left,
                xlabel = \(L\),
                ylabel = {$\EE[S_L]$},
                width=.9\textwidth,
                height=.6\textwidth,
                ymin=0,
                ymax=7
            ]
            \addplot[color=blue,thick] coordinates {
                (0,0)(1,1)
            };
            \addplot [color=blue, thick] coordinates { 
                (1,0)(1.5,0.433231)(2,0.858301)(2.5,0.770878)(3,0.852444)(3.5,0.946821)(4,1.01905)(4.5,1.08943)(5,1.15667)(5.5,1.22086)(6,1.28249)(6.5,1.34185)(7,1.3992)(7.5,1.45473)(8,1.50863)(8.5,1.56105)(9,1.61211)(9.5,1.66194)(10,1.71061)(10.5,1.75823)(11,1.80484)(11.5,1.85054)(12,1.89536)(12.5,1.93936)(13,1.9826)(13.5,2.0251)(14,2.06692)(14.5,2.10808)(15,2.14862)(15.5,2.18858)(16,2.22796)(16.5,2.26681)(17,2.30515)(17.5,2.343)(18,2.38037)(18.5,2.41729)(19,2.45377)(19.5,2.48983)(20,2.52549)(20.5,2.56076)(21,2.59565)(21.5,2.63019)(22,2.66437)(22.5,2.69821)(23,2.73172)(23.5,2.76491)(24,2.7978)(24.5,2.83038)(25,2.86267)(25.5,2.89469)(26,2.92642)(26.5,2.9579)(27,2.98911)(27.5,3.02007)(28,3.05078)(28.5,3.08125)(29,3.11149)(29.5,3.1415)(30,3.17129)(30.5,3.20086)(31,3.23022)(31.5,3.25937)(32,3.28832)(32.5,3.31708)(33,3.34564)(33.5,3.37401)(34,3.40219)(34.5,3.4302)(35,3.45802)(35.5,3.48567)(36,3.51316)(36.5,3.54047)(37,3.56762)(37.5,3.59461)(38,3.62145)(38.5,3.64813)(39,3.67466)(39.5,3.70103)(40,3.72727)(40.5,3.75336)(41,3.77931)(41.5,3.80512)(42,3.83079)(42.5,3.85633)(43,3.88174)(43.5,3.90702)(44,3.93217)(44.5,3.9572)(45,3.98211)(45.5,4.00689)(46,4.03155)(46.5,4.0561)(47,4.08053)(47.5,4.10485)(48,4.12906)(48.5,4.15315)(49,4.17714)(49.5,4.20102)(50,4.22479)(50.5,4.24847)(51,4.27203)(51.5,4.2955)(52,4.31887)(52.5,4.34214)(53,4.36531)(53.5,4.38838)(54,4.41136)(54.5,4.43425)(55,4.45705)(55.5,4.47976)(56,4.50237)(56.5,4.5249)(57,4.54734)(57.5,4.5697)(58,4.59196)(58.5,4.61415)(59,4.63625)(59.5,4.65827)(60,4.68021)(60.5,4.70207)(61,4.72385)(61.5,4.74556)(62,4.76718)(62.5,4.78873)(63,4.8102)(63.5,4.8316)(64,4.85293)(64.5,4.87418)(65,4.89536)(65.5,4.91647)(66,4.9375)(66.5,4.95847)(67,4.97937)(67.5,5.0002)(68,5.02097)(68.5,5.04166)(69,5.06229)(69.5,5.08286)(70,5.10336)(70.5,5.1238)(71,5.14417)(71.5,5.16448)(72,5.18473)(72.5,5.20491)(73,5.22504)(73.5,5.2451)(74,5.26511)(74.5,5.28505)(75,5.30494)(75.5,5.32477)(76,5.34454)(76.5,5.36426)(77,5.38392)(77.5,5.40352)(78,5.42307)(78.5,5.44256)(79,5.462)(79.5,5.48138)(80,5.50071)(80.5,5.51999)(81,5.53921)(81.5,5.55839)(82,5.57751)(82.5,5.59658)(83,5.6156)(83.5,5.63457)(84,5.65349)(84.5,5.67236)(85,5.69118)(85.5,5.70996)(86,5.72868)(86.5,5.74736)(87,5.76599)(87.5,5.78458)(88,5.80311)(88.5,5.8216)(89,5.84005)(89.5,5.85845)(90,5.8768)(90.5,5.89512)(91,5.91338)(91.5,5.9316)(92,5.94978)(92.5,5.96792)(93,5.98601)(93.5,6.00406)(94,6.02207)(94.5,6.04003)(95,6.05796)(95.5,6.07584)(96,6.09368)(96.5,6.11148)(97,6.12924)(97.5,6.14696)(98,6.16464)(98.5,6.18228)(99,6.19988)(99.5,6.21745)
            };
            \end{axis}
        \end{tikzpicture}
        \caption{$\nu(\ell) = 1$.}
    \end{subfigure}%
    \begin{subfigure}{.5\textwidth}
        \begin{tikzpicture}[scale=1]
            \begin{axis} [axis lines = left,
                xlabel = \(L\),
                ylabel = {$\EE[S_L]$},
                width=.9\textwidth,
                height=.6\textwidth,
                ymin=0,
                xmin=0,
                ymax=7
            ]
            \addplot[color=blue,thick] coordinates {
                (0,0)(1,1)
            };
            \addplot [color=blue, thick] coordinates { 
                (1,0)(1.5,0.425582)(2,0.806816)(2.5,0.731057)(3,0.781254)(3.5,0.824283)(4,0.847804)(4.5,0.863477)(5,0.873452)(5.5,0.879678)(6,0.883526)(6.5,0.885886)(7,0.887329)(7.5,0.888207)(8,0.888742)(8.5,0.889066)(9,0.889264)(9.5,0.889383)(10,0.889456)(10.5,0.8895)(11,0.889527)(11.5,0.889543)(12,0.889553)(12.5,0.889559)(13,0.889562)(13.5,0.889564)(14,0.889566)(14.5,0.889566)(15,0.889567)(15.5,0.889567)(16,0.889567)(16.5,0.889568)(17,0.889568)(17.5,0.889568)(18,0.889568)(18.5,0.889568)(19,0.889568)(19.5,0.889568)(20,0.889568)(20.5,0.889568)(21,0.889568)(21.5,0.889568)(22,0.889568)(22.5,0.889568)(23,0.889568)(23.5,0.889568)(24,0.889568)(24.5,0.889568)(25,0.889568)(25.5,0.889568)(26,0.889568)(26.5,0.889568)(27,0.889568)(27.5,0.889568)(28,0.889568)(28.5,0.889568)(29,0.889568)(29.5,0.889568)(30,0.889568)(30.5,0.889568)(31,0.889568)(31.5,0.889568)(32,0.889568)(32.5,0.889568)(33,0.889568)(33.5,0.889568)(34,0.889568)(34.5,0.889568)(35,0.889568)(35.5,0.889568)(36,0.889568)(36.5,0.889568)(37,0.889568)(37.5,0.889568)(38,0.889568)(38.5,0.889568)(39,0.889568)(39.5,0.889568)(40,0.889568)(40.5,0.889568)(41,0.889568)(41.5,0.889568)(42,0.889568)(42.5,0.889568)(43,0.889568)(43.5,0.889568)(44,0.889568)(44.5,0.889568)(45,0.889568)(45.5,0.889568)(46,0.889568)(46.5,0.889568)(47,0.889568)(47.5,0.889568)(48,0.889568)(48.5,0.889568)(49,0.889568)(49.5,0.889568)(50,0.889568)(50.5,0.889568)(51,0.889568)(51.5,0.889568)(52,0.889568)(52.5,0.889568)(53,0.889568)(53.5,0.889568)(54,0.889568)(54.5,0.889568)(55,0.889568)(55.5,0.889568)(56,0.889568)(56.5,0.889568)(57,0.889568)(57.5,0.889568)(58,0.889568)(58.5,0.889568)(59,0.889568)(59.5,0.889568)(60,0.889568)(60.5,0.889568)(61,0.889568)(61.5,0.889568)(62,0.889568)(62.5,0.889568)(63,0.889568)(63.5,0.889568)(64,0.889568)(64.5,0.889568)(65,0.889568)(65.5,0.889568)(66,0.889568)(66.5,0.889568)(67,0.889568)(67.5,0.889568)(68,0.889568)(68.5,0.889568)(69,0.889568)(69.5,0.889568)(70,0.889568)(70.5,0.889568)(71,0.889568)(71.5,0.889568)(72,0.889568)(72.5,0.889568)(73,0.889568)(73.5,0.889568)(74,0.889568)(74.5,0.889568)(75,0.889568)(75.5,0.889568)(76,0.889568)(76.5,0.889568)(77,0.889568)(77.5,0.889568)(78,0.889568)(78.5,0.889568)(79,0.889568)(79.5,0.889568)(80,0.889568)(80.5,0.889568)(81,0.889568)(81.5,0.889568)(82,0.889568)(82.5,0.889568)(83,0.889568)(83.5,0.889568)(84,0.889568)(84.5,0.889568)(85,0.889568)(85.5,0.889568)(86,0.889568)(86.5,0.889568)(87,0.889568)(87.5,0.889568)(88,0.889568)(88.5,0.889568)(89,0.889568)(89.5,0.889568)(90,0.889568)(90.5,0.889568)(91,0.889568)(91.5,0.889568)(92,0.889568)(92.5,0.889568)(93,0.889568)(93.5,0.889568)(94,0.889568)(94.5,0.889568)(95,0.889568)(95.5,0.889568)(96,0.889568)(96.5,0.889568)(97,0.889568)(97.5,0.889568)(98,0.889568)(98.5,0.889568)(99,0.889568)(99.5,0.889568)
            };
            \end{axis}
        \end{tikzpicture}
        \caption{$\nu(\ell) = e^{\ell}$.}
    \end{subfigure}
    \caption{$\EE[S_L]$ grows sublinearly under various divergent ldfs.}
\end{figure}
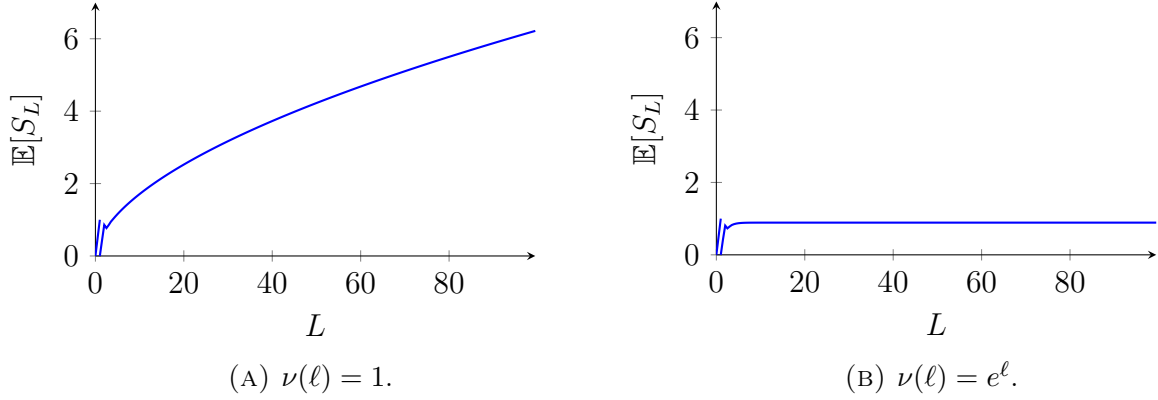

\section{Power Function Distribution}\label{plaw}

In this section, we consider the case when the ldf is given by a power-law function, $\nu(\ell) = (\ell-1)^{\beta-1}$ with $\beta > 0$. This is an example of a divergent ldf, where $\EE[S_L] = o(L)$ per~\cref{div-thm}. However, by using the nice property that the convolution of two power functions is again a power function, we are in fact able to derive much sharper bounds on $\EE[S_L]$.
We will show that $\EE[S_L]$ is intimately connected to a beta distribution associated the the power law. For the rest of this section, we will take $\Beta$ to be the Beta function, given by
\[ \Beta(z_1,z_2) = \int_0^1 t^{z_1-1} (1-t)^{z_2-1} \df{t} = \frac{\Gamma(z_1) \Gamma(z_2)}{\Gamma(z_1 + z_2)}. \]

\begin{obs}\label{plaw-conv}
        Let $\Xi : \RR^{>0} \rightarrow \RR^{>0}$ be a function of $\beta > 0$, defined as the positive solution in $\theta$ to
    \[ \Beta(\beta+1, \theta+1) = \frac{1}{2(\beta+1)}. \]
    Then, the function $\Xi$ is well-defined and at most $1$.
\end{obs}
\begin{proof}
    We first note $ \Beta(\beta+1, \theta+1) = \int_0^1 t^\beta (1-t)^\theta \df{t}.$ The right hand side is strictly decreasing in $\theta$, so $\Beta(\beta+1, \theta+1)$ must also be decreasing in $\theta$. When $\theta = 0$, $\Beta(\beta+1, \theta+1) = \frac{1}{\beta+1} > \frac{1}{2(\beta+1)}$. Meanwhile, when $\theta = 1$, $\Beta(\beta+1, \theta+1) = \frac{1}{(\beta + 1)(\beta + 2)} < \frac{1}{2(\beta+1)}.$ By the continuity of $\Beta$, there must exist $\theta$ with $0 < \theta < 1$ for which $\Beta(\beta+1, \theta+1) = \frac{1}{2 (\beta+1)}$. This equation is equivalent to the condition in the corollary statement. Moreover, because $\Beta(\beta+1, \theta+1)$ is strictly decreasing in $\theta$, our value of $\theta$ is unique, proving that $\Xi$ is well-defined.
\end{proof}

For $\nu(\ell) = (\ell-1)^{\beta-1}$, we will prove $\EE[S_L]$ is bounded below by power functions of the form $L^{\Xi(\beta)-\eps}$, and above by $L^{\Xi(\beta)}$.

\begin{theorem}\label{plaw-asymptotics}
    Fix $\beta > 0$ and consider ldf $\nu(\ell) = (\ell-1)^{\beta-1}$. Then, in the $\nu$-RSA process, for all $\eps > 0$,
    \[ L^{\Xi(\beta)-\eps} \ll \EE[S_L] \leq L^{\Xi(\beta)}. \]
\end{theorem}

\begin{remark}\label{puniform-asymptotic}
    Consider the case when $\beta = 1$, with ldf $\nu(\ell) = 1$. Then, $\Beta(\beta+1, \theta+1) = \frac{1}{2(\beta+1)}$ reduces to the polynomial equation
    \[ (\theta+1) (\theta+2) = 4, \]
    which has unique positive solution $\theta = \Xi(1) = \frac{\sqrt{17}-3}{2} \approx 0.562.$ \cref{plaw-asymptotics} now implies that for all $\eps > 0$,
    \[ L^{(\sqrt{17}-3)/2 - \eps} \ll \EE[S_L] \leq L^{(\sqrt{17}-3)/2}. \]

    Similar processes to the $\nu$-RSA process with the uniform length distribution $\nu(\ell) = 1$ have been studied before. In particular, Coffman et. al. in \cite{coffman} analyze a process in which they park segments with lengths drawn uniformly at random from $[0,L]$. Instead of studying the empty space left at saturation (arbitrarily small segments may be parked in their model, so it never in fact reaches saturation), they study the expected number of parked segments after $n$ attempts to park segments. Using methods different from ours, they derive that the number of parked segments grows as $n^{(\sqrt{17}-3)/2}$. Notably, this exponent is exactly the exponent that we derive for the empty space at saturation in \cref{puniform-asymptotic}, providing an interesting connection between these two related processes.
\end{remark}

\begin{proof}[Proof of \cref{plaw-asymptotics}]
    For simplicity, we take $\nu (\ell) = \frac{1}{\beta} (\ell - 1)^{\beta-1}$ (for $\ell \geq 1$) so that $Z_\nu (L) = (L - 1)^\beta$ (for $L \geq 1$). Define $f_\theta(L) := \frac{\EE[S_L]}{L^\theta}.$ We will prove $f_{\Xi(\beta)}(L) \le 1$, and that for all $\theta< \Xi(\beta)$, we have $f_\theta (L) \ge c_{\theta} > 0$. Let $\mu : [0,1] \rightarrow \RR$ be the pdf of $\text{Beta}(\theta + 1, \beta+ 1)$:
    \[ \mu(t) :=  \frac{t^\theta (1-t)^\beta}{\Beta(\theta+1, \beta+1)}, \]
    so that $\int_0^1 \mu(t) \df{t} = 1$. We now use $\mu$ to derive an integral recurrence relation on $f_\theta$.

    \begin{lemma}\label{plaw-int-recurrence}
        For $\theta, L > 0$,
        \[ f_\theta (L+1) = 2 (\beta+1) \Beta( \theta+1, \beta+1) \cdot \left(\frac{L}{L+1}\right)^{\theta} \cdot \int_0^1 f_\theta(L t) \mu(t) \df{t}. \]
    \end{lemma}
    \begin{proof}
        By \cref{ldf-integral-recurrence} with $L+1$ as the length of the interval,
        \[ \left( \int_0^{L+1} Z_\nu (t) \df{t} \right) \EE[S_{L+1}] = 2 \int_0^{L+1} \EE[S_{t}] Z_\nu (L + 1 - t) \df{t}. \] 
        Substituting in $Z_\nu (L) = (L-1)^\beta$ (when $L > 1$) and $\EE[S_L] = L^\theta f_\theta (L)$ yields
        \[ f_\theta (L+1) = \frac{2(\beta+1) L^{\theta}}{(L+1)^{\theta} } \cdot \int_0^L \frac{t^\theta (L-t)^\beta }{L^{\theta + \beta+1}} f_\theta(t) \df{t}. \]
        Substituting $Lt$ for $t$ and rearranging then proves the lemma.
    \end{proof}
    
    We may now establish the upper bound of the theorem. By \cref{plaw-int-recurrence} with $\theta = \Xi(\beta)$ (and noting $2 (\beta+1) \Beta( \Xi(\beta)+1, \beta+1) = 1$ by the definition of $\Xi$), we have
    \begin{equation}\label{plaw-upperbound}
        f_{\Xi(\beta)}(L+1) = \left(\frac{L}{L+1} \right)^{\Xi(\beta) } \int_0^1 f_{\Xi(\beta)}(Lt) \mu(t) \df{t} < \linechecktext{\int_0^1 f_{\Xi(\beta)}(Lt) \mu(t) \df{t} <} \sup_{0 \leq t \leq L} f_{\Xi(\beta)}(t).
    \end{equation}
    But $ \EE[S_L] = L $ for $L < 1$. Because $\Xi(\beta) < 1$, we have $f_{\Xi(\beta)} (L) = L^{1-\Xi(\beta)} < 1$ when $L < 1$. \cref{plaw-upperbound} then implies that $f_{\Xi(\beta)} (L) < 1$ for all $L$, and thus that $\EE[S_L] \leq L^{\Xi(\beta)}$.

    Now, fix any $\theta < \Xi (\beta)$. Recall $ \Beta( \theta+1, \beta+1 ) $ is decreasing in $\theta$ (c.f. \cref{plaw-conv}), so $2 (\beta+1) \Beta( \theta+1, \beta+1 ) > 1.$ Then, there exists some $\eps > 0$ and some lower bound $L_1 > 0$ such that $L > L_1$ implies
    \[ 2 (\beta+1) \Beta( \theta+1, \beta+1 ) \cdot \left(\frac{L}{L+1}\right)^{\theta} > 1 + \eps,\]
    and thus that $f_\theta( L + 1) > (1+\eps) \int_0^1 f_\theta(Lt) \mu(t) \df{t}.$

    Choose sufficiently large $L_2$ for which $\int_{2/L}^1 \mu(t) \df{t} > \frac{1}{1+ \eps}$ for all $L > L_2$. Let $L^\star = \sup \{ L_1, L_2\}$, so that for $L > L^\star,$
    \begin{equation}\label{plaw-lowerbound}
        f_\theta(L+1) > (1+\eps) \int_{2/L}^1 f_\theta(Lt) \mu(t) \df{t} \geq \inf_{2 \leq t \leq L} f_\theta (t).
    \end{equation}
    Note that \cref{plaw-int-recurrence} implies $f_\theta (L) > 0 $ for $L > 1$. Then, by \cref{plaw-lowerbound}, for all $L > L^\star,$ we have $f_\theta (L) > \sup_{ 2 \leq t \leq L^\star} f_\theta(t),$ where $\sup_{ 2 \leq t \leq L^\star} f_\theta(t)$ is positive.

    Thus, for every $\theta < \Xi(\beta)$, we have shown there is a constant $c_\theta > 0$ such that $f_\theta(L) > c_\theta$  for sufficiently large $L$, or equivalently that $\EE[S_L] > c_\theta L^\theta$. This in fact implies $\EE[S_L] \gg L^\theta$ for all $\theta < \Xi(\beta)$, completing the proof of the theorem.
\end{proof}

We may also derive the following more tractable form of $\Xi(\beta)$ when $\beta$ is an integer.

\begin{corollary}\label{plaw-int-asymptotics}
    If $\beta > 0$ is an integer, then let $\theta$ be the unique positive root to
    $$\prod_{i=1}^{\beta+1} (\theta + i) = 2 (\beta + 1)!.$$
    Consider ldf $\nu (\ell) = (\ell - 1)^{\beta-1}$. Then, in the $\nu$-RSA process, for all $\eps > 0 $,
    \[ L^{\theta-\eps} \ll \EE[S_L] \leq L^{\theta}. \]
\end{corollary}
\begin{proof}
    When $\beta$ is an integer, the equation $\Beta( \theta+1, \beta+1 ) = \frac{1}{2(\beta+1)} $  reduces to the polynomial $\prod_{i=1}^{\beta+1} (\theta + i) = 2 (\beta + 1)!$. The uniqueness of $\theta$ follows from \cref{plaw-conv}, and then \cref{plaw-int-asymptotics} immediately follows from \cref{plaw-asymptotics}.
\end{proof}

\bibliography{bib}
\bibliographystyle{unsrt}

\end{document}